\documentclass[a4paper,11pt]{article}
\usepackage{amsfonts}
\usepackage{bbm}
\usepackage{mathrsfs}
\usepackage{amsmath,amsthm,amssymb,amscd}
\usepackage[center]{titlesec}
\newtheorem{theo}{Theorem}[section]
\newtheorem{prop}[theo]{Proposition}
\newtheorem{lem}[theo]{Lemma}
\newtheorem{rem}[theo]{Remark}
\newtheorem{col}[theo]{Corollary}
\newtheorem{defi}[theo]{Definition}
\newcommand{\bp}{\begin{proof}}
\newcommand{\ep}{\end{proof}}

\begin{document}
\setlength{\baselineskip}{13pt} \pagestyle{myheadings}
\title{{ On closed almost complex four manifolds}\thanks{ Supported by
 NSFC Grants 12171417(Hongyu Wang), 11971112 (Ken Wang).}}
\author{{\large Hongyu WANG \ \ \ \
 Ken WANG \ \ \ \
 Peng ZHU }\\
}
\date{}
\maketitle
\begin{quote}
\noindent {\bf Abstract.} {\small In this paper, by using weakly $\widetilde{\mathcal{D}}^+_J$ (resp. $\mathcal{D}^+_J$)-closed technique firstly introduced
by Tan,Wang, Zhou and Zhu, we will give
a characterization of tamed and weakened tamed four-manifolds, and an almost K\"{a}hler version of Nakai-Moishezon criterion for almost complex four-manifolds.}
\end{quote}

\section{Introduction}
Suppose that $M$ is a closed, oriented smooth $4$-manifold and $\omega$ is a symplectic form on $M$ that is compatible with
the orientation. An endomorphism, $J$, of $TM$ is said to be an almost complex structure when $J^2=-\mathrm{id}_{TM}$. Such an almost
complex is said to be tamed by $\omega$ when the bilinear form $\omega(\cdot,J\cdot)$ is positive definite. The almost complex
structure $J$ is said to be compatible (or calibrate) with $\omega$ when this same bilinear form is also symmetric \cite{McSa}.
Gromov \cite{Gr} observed that tamed almost complex structures and also compatible structures always exist.

  Note that there are topological obstructions to the existence of almost complex structures on an even dimensional manifold.
  For a closed 4-manifold, a necessary condition is that $1-b^1+b^+$ be even \cite{McSa},
   where $b^1$ is the first Betti number and $b^+$ is the number of positive eigenvalues of
   the quadratic form on $H^2(M;\mathbb{R})$ defined by the cup product, hence the condition is either $b^1$ be even and $b^+$ odd,
    or $b^1$ be odd and $b^+$ even.
  It is a well-known fact (which is the Kodaira conjecture \cite{BaPeVa,MoKo}) that any closed complex surface with $b^1$ even is K\"{a}hler.
  The direct proofs have been given by N. Buchdahl \cite{Bu} and A. Lamari \cite{La1}.
  R. Harvey and H. B. Lawson, Jr. (Theorems 26 and 38 in \cite{HaLa}) have proved that for any closed complex surface $(M,J)$ with $b^1$ even,
   there exists a symplectic form $\omega$ on $M$ by which $J$ is tamed.
   S.K. Donaldson \cite{Do} posed the following question:\\

   If an almost complex structure on a closed $4$-manifold is tamed by a given symplectic form $\omega$,
must it be compatible with a new symplectic form?\\

Thus, Donaldson's question on tamed almost complex 4-manifolds (in particular, $b^+=1$) is related to the
   Kodaira conjecture for complex surfaces \cite{Do,DLZ1}.\\

On a closed almost complex $4$-manifold $(M,J)$, Li and Zhang \cite{LiZh} introduced subgroups $H_J^+$ and $H_J^-$ of the real degree 2
de Rham cohomology $H^2(M;\mathbb{R})$, as the set of cohomology classes which can be represented by $J$-invariant and $J$-anti-invariant $2$-real
forms. Let us denote by $h_J^+$ and $h_J^-$ the dimensions of $H_J^+$ and $H_J^-$, respectively. In \cite{DLZ1}, it was proved that for a closed almost
complex $4$-manifold $(M,J)$,
$$H^2(M;\mathbb{R})=H_J^+\oplus H_J^-.$$

Tan, Wang, Zhou and Zhu \cite[Theorem 1.1]{TaWaZhZh} solved Donaldson tameness question with the condition of $h_J^-=b^+-1$.
The approach is along the lines given by Buchdahl \cite{Bu} for a unified proof of the Kodaira conjecture:\\

{\it  Let $M$ be a closed symplectic 4-manifold with symplectic form $\omega$.
  Suppose that $J$ is an $\omega$-tamed almost complex structure on $M$ and $h^-_J=b^+-1$.
  Then there exists a new symplectic form $\Omega$ that is compatible with $J$.}\\

It is worth remarking that the proof of Theorem 1.1 in \cite{TaWaZhZh} is similar to the proof
 of Theorem 11 in \cite{Bu}. Without loss of generality, we can assume that
 $$\omega=F+d_J^-(v+\bar{v})$$
  is the $J$-taming symplectic form on the almost Hermitian $4$-manifold $(M,g,J,F)$, where
 $$F:=P_J^+\omega\in\Omega_{J}^{1,1}(M)$$
  is the fundamental $2$-form, $v\in\Omega_{J}^{0,1}(M)$, then
 $$\widetilde{\omega}=\omega-d(v+\bar{v})=F-d_J^+(v+\bar{v})$$
  is in $H_J^+\cap H_g^+$, where $H_g^+$ is the subgroup of $H^2(M;\mathbb{R})$, as the set of cohomology classes which can be represented by
  $g$-self-dual 2-forms.
 Since $h_J^-=b^+-1$, it follows that there exists $t_1>0$ such that
  $$\widetilde{\omega}>t_1F,$$
  in the sense of currents and hence there exists a symplectic form $\Omega$ that is compatible with the almost complex structure $J$.
  In fact, arguing as in the proof of Theorem 12 of \cite{Bu}, we can prove that there exists $u\in\Omega_J^{0,1}(M)$ such that
   $$\widetilde{\omega}=F-d_J^+(u+\bar{u})$$
   is a $d$-closed positive $(1,1)$-form on $M$, that is, $\widetilde{\omega}$ is an almost K\"{a}hler form on $M$.
The key technique is the introduction of the operator:
 $$\widetilde{\mathcal{D}}^+_J (\ {\rm resp.}\ \mathcal{D}^+_J) : L^2_2(M)_0\longrightarrow\Lambda^{1,1}_J\otimes L^2(M)$$
 which is similar to the operator $\partial\bar{\partial}$ as in classical complex analysis \cite{TaWaZhZh}, and study the space of weakly $\widetilde{\mathcal{D}}^+_J$ (resp. $\mathcal{D}^+_J$)-
 closed $J$-$(1,1)$ forms.\\

 In this paper, we will study the characterization of closed almost complex $4$-manifolds and Nakai-Moishezhon criterion for almost complex $4$-manifolds
 by using weakly $\widetilde{\mathcal{D}}^+_J$ (resp. $\mathcal{D}^+_J$)-closed technique which is developed in \cite{TaWaZhZh} and give main results as follows:\\

 Similar to Theorem 12 in \cite{Bu}, we have:\vspace{0.2cm}\\
 {\bf Theorem 4.3.}
{\it Suppose that $(M,J)$ is a closed tamed almost complex $4$-manifold with $h_J^-=b^+-1$, where $$\omega=F+d_J^-(v+\bar{v})$$ is the $J$-taming
symplectic $2$-form on $M$, $v\in \Omega_J^{0,1}(M)$. Then $$\widetilde{\omega}=F-d_J^+(v+\bar{v})$$
is a $J$-almost K\"{a}hler form on $M$ modulo the image of ${\widetilde{\mathcal{D}}}^+_J$.}\\

Similar to Theorem 38 of Harvey and Lawson \cite{HaLa0}, a weakened tamed symplectic structure can be characterized as follows:\vspace{0.2cm}\\
{\bf Theorem 4.7.}
{\it Suppose that $(M,g,J,F)$ is a closed almost Hermitian $4$-manifold. If $F$ is  weakly
 ${\textit{D}}_J^+$-closed, then there exist
 $$u_F\in \wedge_J^{0,1}\otimes L_1^2(M)\ \ and \ \  u_F'\in\Omega_J^{0,1}(M)$$
  such that $*_g\widetilde{F}$ is $d$-closed, where
  $$\widetilde{F}:=F+d_J^+(u_F+\bar{u}_F)=C_FF+d_g^-(u_F'+\bar{u}_F'),$$
  $$*_g\widetilde{F}=C_FF-d_g^-(u_F'+\bar{u}_F').$$
Furthermore, if $h_J^-=b^+-1$ and
$$\int_M(*_g\widetilde{F})^2>0,$$
 then there exists a $J$-almost K\"{a}hler form on $(M,J)$.}\\

As in compact complex surfaces, we give a Nakai-Moishezon criterion for tamed almost complex $4$-manifolds:\vspace{0.2cm}\\
 {\bf Theorem 5.1.}
 {\it Suppose that $(M,g,J,F)$ is a closed almost Hermitian 4-manifold with $h^-_J=b^+-1$
 tamed by a symplectic form $\omega=F+d_J^-(v+\bar{v})$, where $v\in \Omega_J^{0,1}(M)$.
 If $\chi\in (\wedge_J^{1,1}\otimes L^2(M))_{\widetilde{w}}$ satisfies
 $$\int_M\chi^2>0, \int_MF\wedge\chi>0,$$
 and
 $$\int_D\chi>0$$
 for each irreducible $J$-curve (i.e. $J$-invariant homology class in $H_+^J(M)$) $D$ on $M$. Then $\chi$ is homologous to
  a smooth closed positive $J$-$(1,1)$ form modulo the image of $d_J^+$, and homologous to a smooth positive $J$-$(1,1)$ form
  modulo the image of $\widetilde{\mathcal{D}}_J^+$.}\\

  In Theorem 5.1, the classical Nakai-Moishezon criterion for a closed tamed almost complex $4$-manifold was generalized to yield a characterization
   of the set of classes $H_J^+$ which can be represented by an almost K\"{a}hler form. Under the assumption that $h_J^-=b^+-1$, this result is further
   generalized to the case of weakly ${\mathcal{D}}_J^+$-closed modulo ${\mathcal{D}}_J^+$-exact $(1,1)$ form. The following theorem  is to
   demonstrate that the assumption is $h_J^+=b^+$:\vspace{0.2cm}\\
{\bf Theorem 5.3.}
{\it Let $(M,J)$ be a closed almost complex $4$-manifold equipped with a weakly ${\mathcal{D}}_J^+$-closed the fundamental $2$-form $F$
 and $h_J^-=b^+$. Let $\varphi$ be a smooth real weakly ${\mathcal{D}}_J^+$-closed $J$-$(1,1)$ form satisfying
 $$\int_M\varphi\wedge\varphi>0,\int_M\varphi\wedge F>0,\int_D\varphi>0,$$ for every irreducible $J$-curve $D\subset M$ with $D\cdot D<0.$
 If there is an effective non-zero integral $J$-curve $E$ on $M$ with $E\cdot E=0$, then there is a smooth function $f$ on $M$ such that
 $\varphi+{\mathcal{D}}_J^+(f)$} is positive.\\

 The remainder of the paper is organized as follows:\\

In section 2, we introduce the operator $\mathcal{D}_J^+$ and $\widetilde{\mathcal{D}}^+_J$ and their basic properties.

Section 3 is devoted to the intersection pairing on weakly $\widetilde{\mathcal{D}}_J^+$ (resp. $\mathcal{D}_J^+$)-closed $J$-(1,1) forms on almost complex manifolds.

In Section 4, we give a characterization of closed almost complex $4$-manifolds which admits tamed and weaken tamed symplectic structures (Theorems 4.3 and 4.7).

In section 5, we prove tamed version and weak version of  Nakai-Moishezon criterion for almost complex $4$-manifolds (Theorem \ref{th51} and Theorem \ref{th53}).

 Section 6 is an appendix on $(\mathcal{\widetilde{W}},d_J^-)$-problem and $(\mathcal{W},d_J^-)$-problem as $\bar{\partial}$-problem in classical complex analysis.

\section{Preliminaries}
 \setcounter{equation}{0}
 Suppose that $M$ is  an almost complex manifold with almost complex structure $J$,
   then for any $x\in M$, $T_x(M)\otimes_\mathbb{R}\mathbb{C}$ which is the complexification of $T_x(M)$,
   has the following decomposition (cf. \cite{BaPeVa,LiZh}):
   \begin{equation}\label{0010}
  T_x(M)\otimes_\mathbb{R}\mathbb{C}=T^{1,0}_x+T^{0,1}_x,
  \end{equation}
  where $T^{1,0}_x$ and $T^{0,1}_x$ are the eigenspaces of $J$ corresponding to the eigenvalues $\sqrt{-1}$ and $-\sqrt{-1}$, respectively.
   Let $TM\otimes_\mathbb{R}\mathbb{C}$ be the complexification of the tangent bundle.
   Similarly, let $T^*M\otimes_{\mathbb{R}}\mathbb{C}$ denote the complexification of the cotangent bundle $T^*M$,
   $J$ can act on $T^*M\otimes_{\mathbb{R}}\mathbb{C}$ as follows:
   $$\forall \alpha\in T^*M\otimes_{\mathbb{R}}\mathbb{C},\,\,\, J\alpha(\cdot)=-\alpha(J\cdot).$$
   Hence $T^*M\otimes_{\mathbb{R}}\mathbb{C}$ has the following decomposition according to the eigenvalues $\mp\sqrt{-1}$:
   \begin{equation}\label{0011}
   T^*M\otimes_{\mathbb{R}}\mathbb{C}=\Lambda^{1,0}_J\oplus\Lambda^{0,1}_J.
   \end{equation}
   One can form exterior bundle $\Lambda^{p,q}_J=\Lambda^p\Lambda^{1,0}_J\otimes\Lambda^q\Lambda^{0,1}_J$.
   Let $\Omega^{p,q}_J(M)$ denote the space of $C^\infty$ sections of the bundle $\Lambda^{p,q}_J$.
   The exterior differential operator acts on $\Omega^{p,q}_J$ as follows:
   \begin{equation*}
   d\Omega^{p,q}_J\subset\Omega^{p-1,q+2}_J+\Omega^{p+1,q}_J+\Omega^{p,q+1}_J+\Omega^{p+2,q-1}_J.
   \end{equation*}
  Hence, $d$ has the following decomposition:
  \begin{equation}\label{0012}
  d=A_J\oplus\partial_J\oplus\bar{\partial}_J\oplus\bar{A}_J.
  \end{equation}
  For more details, see \cite[Section 2]{TaWaZhZh}.

   Suppose that $(M,J)$ is an almost complex 4-manifold.
  One can construct a $J$-invariant Riemannian metric $g$ on $M$, namely,
  $$g(JX,JY)=g(X,Y)$$
   for all tangent vector fields $X$ and $Y$ on $M$.
  Such a metric $g$ is called an almost Hermitian metric (real) on $(M,J)$.
   This then in turn gives a $J$-compatible nondegenerate 2-form $F$ on $M$ by $F(X,Y)=g(JX,Y)$, called the fundamental 2-form.
  Such a quadruple $(M,g,J,F)$ is called an almost Hermitian 4-manifold.
  Thus an almost Hermitian structure on $M$ is a triple $(g,J,F)$.
    By using almost Hermitian structure $(g,J,F)$, one can define a volume form $d\mu_g=F^2/2$ with $$\int_Md\mu_g=1$$
   by rescaling in the conformal equivalent class $[g]$.
  If the 2-form $F$ is $d$-closed, then the triple $(g,J,F)$ is called an almost K\"{a}hler structure.
  When $F$ is closed and $J$ is integrable, the $(g,J,F)$ defines a K\"{a}hler structure on $M$ \cite{McSa}.
  Although $M$ need not admit a symplectic condition (that is, $dF=0$),
    P. Gauduchon \cite{Ga} showed that  for a closed almost Hermitian 4-manifold $(M,g,J,F)$ there is a conformal rescaling of the metric $g$,
     unique up to positive constant, such that the associated form satisfies $\bar{\partial}_J\partial_J F=0$.
     This metric is called Gauduchon metric.\\

  Let $\Omega^2_\mathbb{R}(M)$ denote the space of real smooth 2-forms on $M$.
   The almost complex structure $J$ acts on $\Omega^2_\mathbb{R}(M)$ as an involution by $$ \mathcal{J}:\alpha(\cdot,\cdot)\mapsto\alpha(J\cdot,J\cdot),$$
 which satisfies $\mathcal{J}^2=\mathrm{id}$ on $\Omega^2_\mathbb{R}(M)$.  Thus we have the splitting into $J$-invariant and $J$-anti-invariant 2-forms respectively
  \begin{equation}\label{2eq14}
  \Lambda^2_\mathbb{R}=\Lambda^+_J\oplus\Lambda^-_J,
  \end{equation}
   where the bundles $\Lambda^{\pm}_J$ are defined by
  $$
  \Lambda^{\pm}_J=\{\alpha\in\Lambda^2_\mathbb{R}\mid \alpha(J\cdot,J\cdot)=\pm\alpha(\cdot,\cdot)\}.
  $$
  We will denote by
  $$\Omega^+_J=\Omega_J^{1,1}$$
   and
   $$\Omega^-_J=(\Omega_J^{2,0}+\Omega_J^{0,2})_{\mathbb{R}},$$
   respectively, the $C^\infty$ sections of the bundles $\Lambda^+_J$ and $\Lambda^-_J$.
  For $\alpha\in\Omega^2_\mathbb{R}(M)$, denote by $\alpha^+_J$ and $\alpha^-_J$,
   respectively, the $J$-invariant and $J$-anti-invariant components of $\alpha$ with respect to the decomposition (\ref{2eq14}).
   We will also use the notation $\mathcal{Z}^2_\mathbb{R}$ for the space of real closed 2-forms on $M$
   and $\mathcal{Z}^\pm_J=\mathcal{Z}^2_\mathbb{R}\cap\Omega^\pm_J$ for the corresponding projections. Define
   \begin{equation}
   d^{\pm}_J=P^{\pm}_Jd: \Omega^1_\mathbb{R}(M)\rightarrow\Omega^{\pm}_J(M),
 \end{equation}
 where $$P_J^{\pm}=\frac{1}{2}(1\pm \mathcal{J})$$ are algebraic projections on $\Omega^2_\mathbb{R}(M)$.
   Li and Zhang have defined in \cite{LiZh} the $J$-invariant and $J$-anti-invariant cohomology subgroups $H^\pm_J$ of $H^2(M;\mathbb{R})$ as follows:
  $$H^\pm_J=\{\mathfrak{a}\in H^2(M;\mathbb{R})\mid \exists \alpha\in\mathcal{Z}^\pm_J \,\,\,{\rm such} \,\,\, {\rm that} \,\,\, [\alpha]=\mathfrak{a}\};$$
  $J$ is said to be $C^\infty$-pure if $H^+_J\cap H^-_J=\{0\}$, $C^\infty$-full if $H^+_J+H^-_J=H^2(M;\mathbb{R})$.
  $J$ is $C^\infty$-pure and full if and only if $H^2(M;\mathbb{R})=H^+_J\oplus H^-_J$. Draghici, Li and Zhang have the following result \cite[Theorem 2.2]{DLZ1}:
   \begin{prop}\label{2p1}
 If $(M,J)$ is a  closed almost complex 4-manifold,
   then the almost complex structure $J$ on $M$ is $C^\infty$-pure and full.
   Thus, there is a direct sum decomposition
   $$H^2(M;\mathbb{R})=H^+_J\oplus H^-_J.$$
   Let $h^+_J$ and $h^-_J$ denote the dimensions of $H^+_J$ and $H^-_J$, respectively.
   Then we have $b^2=h^+_J+h^-_J$, where $b^2$ is the second Betti number.
   \end{prop}

   It is worth remarking that on a closed almost complex $4$-manifold $(M,J)$,
   we can define the $J$-invariant and $J$-anti-invariant homology subgroups $H_{\pm}^J(M)$ of $H_2(M;\mathbb{R})$ with
   the following relation:
   $$H_2(M;\mathbb{R})=H_+^J(M)\oplus H_-^J(M).$$
   Also, $H_{\pm}^J(M)$ and $H^{\pm}_J(M)$ have the same dimension and are dual to each other. More details, see \cite{DLZ2,LiZh}.

      Since $(M,g,J,F)$ is a closed almost Hermitian 4-manifold, Hodge star operator $*_g$ with respect to the metric $g$ gives the self-dual,
   anti-self-dual decomposition of the bundle of 2-forms \cite{DoKr}:
  \begin{equation}\label{2eq15}
  \Lambda^2_\mathbb{R}=\Lambda^+_g\oplus\Lambda^-_g.
  \end{equation}
   We denote $\Omega^\pm_g$ as the spaces of smooth sections of $\Lambda^\pm_g$.
   Since the Hodge-de Rham Laplacian
   $$\Delta_g=dd^*+d^*d:\Omega^2(M)\rightarrow\Omega^2(M),$$
   where $d^*=-*_gd*_g$ is the codifferential operator with respect to the metric $g$,
    commutes with Hodge star operator $*_g$, the decomposition (\ref{2eq15}) holds for the space $\mathcal{H}_g$ of harmonic 2-forms as well.
    By Hodge theory, this induces a cohomology decomposition by the metric $g$:
  $$ \mathcal{H}_g=\mathcal{H}_g^+\oplus\mathcal{H}_g^-.$$

  Let
    \begin{equation}
  d^\pm_g=P_g^{\pm}d:\Omega^1_\mathbb{R}\rightarrow\Omega^\pm_g
 \end{equation}
 be the first-order differential operator defined by the composite of the exterior derivative
   $d: \Omega^1_\mathbb{R}\rightarrow\Omega^2_\mathbb{R}$ with the algebraic projections
   $$P^\pm_g=\frac{1}{2}(1\pm*_g)$$
   from $\Omega^2_\mathbb{R}$ to $\Omega^\pm_g$.
   So we can get the following Hodge decompositions \cite{DoKr}:
   \begin{equation*}
   \Omega^+_g=\mathcal{H}^+_g\oplus d^+_g(\Omega_{\mathbb{R}}^1), \,\,\, \Omega^-_g=\mathcal{H}^-_g\oplus d^-_g(\Omega_{\mathbb{R}}^1).
  \end{equation*}
    Note that
   \begin{equation}\label{2eq16'}
  d^\pm_gd^*:\Omega^\pm_g\rightarrow\Omega^\pm_g
   \end{equation}
  are self-adjoint strongly elliptic operators and $\ker d^\pm_gd^*=\mathcal{H}^\pm_g$.
  If $d^+_gu$ is $d$-closed, that is, $dd^+_gu=0$, then
  $$
   0=\int_Mdd^+_gu\wedge u=-\int_Md^+_gu\wedge du=-\int_M|d^+_gu|^2,
  $$
  so $d^+_gu=0$. Similarly, for any $u\in \Omega^1_\mathbb{R}$, if $d^+_gu=0$,
    \begin{equation*}
   0=\int_Mdu\wedge du=\int_M|d^+_gu|^2-\int_M|d^-_gu|^2=-\int_M|d^-_gu|^2,
  \end{equation*}
  so $d^-_gu=0$ too, therefore we can get $du=0$ \cite{DoKr}.

  Define
  $$
    H^\pm_g=\{ \mathfrak{a}\in H^2(M;\mathbb{R})\mid\exists\alpha\in\mathcal{Z}^\pm_g:=\mathcal{Z}_{\mathbb{R}}^2\cap\Omega^\pm_g
     \,\,\, {\rm such} \,\,\, {\rm that} \,\,\, \mathfrak{a}=[\alpha]\}.
   $$
    There are the following relations:
   \begin{equation}\label{2eq18}
  \Lambda^+_J=\mathbb{R}\cdot F\oplus\Lambda^-_g,\,\,\,
  \Lambda^+_g=\mathbb{R}\cdot F\oplus\Lambda^-_J,
  \end{equation}
  \begin{equation}\label{2eq19}
  \Lambda^+_J\cap\Lambda^+_g=\mathbb{R}\cdot F, \,\,\, \Lambda^-_J\cap\Lambda^-_g=\{0\}.
  \end{equation}
  It is easy to see that $H^-_J\subset H^+_g$ and $H^-_g\subset H^+_J$ (cf. \cite{DLZ2,TWZZ}).

   Let $b^+$ denote the self-dual Betti number, and $b^-$ the anti-self-dual Betti number of $M$, hence $b^2=b^++b^-$.
   Thus, for a closed almost Hermitian 4-manifold $(M,g,J,F)$, we have (c.f. Tan, Wang, Zhang and Zhu \cite{TWZZ}):
   $$
   \mathcal{Z}^-_J\subset\mathcal{Z}^+_g, \,\,\, \mathcal{Z}^-_g\subset\mathcal{Z}^+_J, b^++b^-=h^+_J+h^-_J, \,\,\, h^+_J\geq b^-, \,\,\, 0\leq h^-_J\leq b^+.
   $$
   If $(M,g,J,\omega)$ is a closed tamed $4$-manifold, then \cite{TaWaZh}
   $$0\leq h_J^-\leq b^+-1.$$
   It is worth remarking that $b^2$, $b^+$ and $b^-$ are topological invariants, but $h_J^+$ and $h_J^-$ are not topological invariants \cite{TWZZ,TaWaZh}.\\

   M. Lejmi \cite{Le0,Le1} recognizes $\mathcal{Z}^-_J$ as the kernel of an elliptic operator on $\Omega^-_J$ and gives the following result \cite[Lemma 4.1]{Le1}:
  \begin{lem} \label{2130}
   Let $(M,g,J,F)$ be a closed almost Hermitian 4-manifold.
  Let operator
  $$P:\Omega^-_J\rightarrow\Omega^-_J$$
   be defined by
  \begin{eqnarray*}
  P(\psi)=P^-_J(dd^*\psi),
  \end{eqnarray*}
  for $\psi\in\Omega^-_J$.
   Then $P$ is a self-adjoint strongly elliptic linear operator with kernel the $g$-self-dual-harmonic, $J$-anti-invariant 2-forms.
  Hence,
  \begin{eqnarray*}
  \Omega^-_J=\ker P\oplus d^-_J\Omega^1_\mathbb{R}.
  \end{eqnarray*}
   \end{lem}

In the case of a closed almost Hermitian 4-manifold $(M,g,J,F)$, the splitting of 2-forms into types is compatible with
the splitting of forms into self-dual and anti-self-dual parts from the underlying Riemannian metric induced by $g$, namely,
$$\wedge_g^+=\mathbb{R}\cdot F+(\wedge_J^{2,0}+\wedge_J^{0,2})_{\mathbb{R}},\wedge_g^-=ker F\wedge:\wedge_J^{1,1}\longrightarrow\wedge_J^{2,2}.$$
Let $\wedge$ be the adjoint operator of
$$L:f\longrightarrow fF,$$
 for $f\in C^{\infty}(M).$ Hence, for a real $J$-$(1,1)$ form $\psi$,
 \begin{equation}\label{40001}
 F\wedge\psi=(\wedge\psi)F^2/2,
 \end{equation}
 \begin{equation}\label{40002}
 *_g\psi=(\wedge\psi)F-\psi,
 \end{equation}
 \begin{equation}\label{40003}
 |\psi|_g^2dvol_g=(\wedge\psi)^2F^2/2-\psi^2.
 \end{equation}
 \begin{lem}(c.f. \cite[Lemma 1]{Bu}, \cite[Lemma 3.1]{TaWaZhZh})\label{311}
  Suppose that $(M,g,J,F)$ is a closed almost Hermitian 4-manifold.
  Then $$d^+_J: \Lambda^1_\mathbb{R}\otimes L^2_1(M)\longrightarrow\Lambda^{1,1}_J\otimes L^2(M)$$ has closed range.
  \end{lem}
  \begin{proof}
 Let $u_i\in\Lambda^{0,1}_J\otimes L_1^2(M)$, then $\{u_i+\bar{u}_i\}$ is a sequence of  real 1-forms on $M$ with coefficients in $L_1^2$
 such that
 $$\psi_i=d_J^+(u_i+\bar{u}_i)=\partial_J u_i+\bar{\partial}_J\bar{u}_i$$
converges in $L^2$ to some $\psi\in\Lambda^{1,1}_J\otimes L^2(M)$.  By smoothing and  diagonalising \cite{Au},
  it can be assumed without loss of generality that $\psi_i$ is smooth for each $i$.
  Since
    \begin{eqnarray*}
    d(u_i+\bar{u}_i)=d_J^+(u_i+\bar{u}_i)+ d_J^-(u_i+\bar{u}_i)=\psi_i+d_J^-(u_i+\bar{u}_i),
     \end{eqnarray*}
      \begin{eqnarray*}
    d(u_i+\bar{u}_i)\wedge d(u_i+\bar{u}_i)=\psi_i^2+(d_J^-(u_i+\bar{u}_i))^2,
    \end{eqnarray*}
    and
    \begin{eqnarray*}
    d(u_i+\bar{u}_i)\wedge*_g d(u_i+\bar{u}_i)=\psi_i\wedge*_g\psi_i+(d_J^-(u_i+\bar{u}_i))^2,
    \end{eqnarray*}
  by \eqref{40001}-\eqref{40003}, using Stokes Theorem,
   \begin{eqnarray*}
&\|\psi_i\|^2&=\int_M(\wedge \psi_i )^2\frac{F^2}{2}-\int_M\psi_i^2\\
&=&\int_M(\wedge\psi_i)^2\frac{F^2}{2}+\int_M(d_J^-(u_i+\bar{u}_i))^2\\
&=&\|\wedge\psi_i\|^2+\|(d_J^-(u_i+\bar{u}_i))\|^2,
\end{eqnarray*}
and
 \begin{eqnarray*}
\|d(u_i+\bar{u}_i)\|^2&=&\int_Md(u_i+\bar{u}_i)\wedge*_gd(u_i+\bar{u}_i)\\
&=&\|\psi_i\|^2+\|(d_J^-(u_i+\bar{u}_i))\|^2\\
&\leq&2\|\psi_i\|^2.
\end{eqnarray*}
  So it follows that $d(u_i+\bar{u}_i)$ is bounded in $L^2$. Note that
   $$\Omega^1(M)=\mathcal{}H^1\oplus \mathrm{Im} d\oplus \mathrm{Im} d^*.$$
 Let $\widetilde{u}_i+\widetilde{\bar{u}}_i$ be the $L^2$-projection of $u_i+\bar{u}_i$  to $\mathrm{Im} d^*$, that is,
  $\widetilde{u}_i+\widetilde{\bar{u}}_i\in \mathrm{Im} d^*.$
Obviously, $$d^*(\widetilde{u}_i+\widetilde{\bar{u}}_i)=0,$$
  and $\widetilde{u}_i+\widetilde{\bar{u}}_i$ is perpendicular to the harmonic 1-forms \cite{DoKr}.
 Hence there is a constant $C$ such that
 \begin{equation*}
 \|\widetilde{u}_i+\widetilde{\bar{u}}_i\|_{L^2_1(M)}\leq C(\|d(\widetilde{u}_i+\widetilde{\bar{u}}_i)\|_{L^2(M)}+\|d^*(\widetilde{u}_i+\widetilde{\bar{u}}_i)\|_{L^2(M)})= \mathrm{Const}.,
 \end{equation*}
 so a subsequence of the sequence $\{\widetilde{u}_i+\widetilde{\bar{u}}_i\}$ converges weakly in $L^2_1$ to some $\widetilde{u}+\widetilde{\bar{u}}\in\Lambda^1_\mathbb{R}\otimes L^2_1(M)$. Since
 $$d^+_J(\widetilde{u}_i+\widetilde{\bar{u}}_i)=\psi_i,$$
 it follows that
  $$d^+_J(\widetilde{u}+\widetilde{\bar{u}})=\psi.$$
 \end{proof}
Let $(M,J)$ be an almost complex $4$-manifold. We can study $J$-$(1,1)$ currents on $M$ (c.f. \cite{HaLa0,HaLa,LiZh1,LiZh2,TaWaZhZh,Zh}).
Similar to Proposition 25 in \cite{HaLa0} for compact complex surfaces, we have the following proposition:
\begin{prop}\label{204}
Assume that $T$ is a real $J$-(1,1) current on a closed almost complex $4$-manifold $(M,J)$, and $T=d_J^+S$ for some real current $S=S^{1,0}+S^{0,1}$,
if $dT=0$ in the sense of curents, then, in fact, $T=dS$, that is, $d_J^-S=0$.
\end{prop}
\bp
$$T=d_J^+S=\bar{\partial}_JS^{1,0}+\partial_JS^{0,1}=dS-d_J^-S.$$
It follows that
$$0=dT=-dd_J^-S,$$
in the sense of currents.
Hence, by Stokes Theorem,
\begin{eqnarray}\label{5001}
0&=&\int_Md(S\wedge d_J^-S)\nonumber\\
&=&\int_Md_J^-S\wedge d_J^-S.
\end{eqnarray}
We can construct an almost Hermitian structure, $(g,J,F)$, on $M$. It is easy to see that $\Omega_J^-\subset \Omega_g^+$.
Hence, by \eqref{2eq18},
$$d_J^-S\wedge d_J^-S=\|d_J^-S\|_g^2,$$
where $\|\cdot\|_g$ is the norm of $\Omega_{\mathbb{R}}^2$ with respect to  the metric $g$.
 By \eqref{5001}, it follows that $d_J^-S=0$.
\ep
 By using the operator $P$
defined in Lemma \ref{2130}, Tan, Wang, Zhou and Zhu \cite{TaWaZhZh}  introduced the $\mathcal{D}^+_J$ operator:
\begin{defi}\label{0103}
 Let $(M,g,J,F)$ be a closed almost Hermitian 4-manifold. Denote by
  $$L_2^2(M)_0:=\{f\in L_2^2(M)|\int_Mfd\mu_g=0\}.$$
  Set
   $$\mathcal{W}: L^2_2(M)_0\longrightarrow\Lambda^1_\mathbb{R}\otimes L^2_1(M),$$
   $$\mathcal{W}(f)=Jdf+d^*(\eta^1_f+\overline{\eta}^1_f),$$
where $\eta^1_f\in\Lambda^{0,2}_J\otimes L^2_2(M)$ satisfies $$d^-_JJdf+d^-_Jd^*(\eta^1_f+\overline{\eta}^1_f)=0.$$
  Obviously,
   $$d^-_J\mathcal{W}(f)=0.$$
Define
$$\mathcal{D}^+_J: L^2_2(M)_0\longrightarrow\Lambda^{1,1}_J\otimes L^2(M),$$
$$\mathcal{D}^+_J(f)=d\mathcal{W}(f).$$
$\psi\in\Lambda^{1,1}_J\otimes L^2(M)$ is said to be weakly $\mathcal{D}^+_J$-closed if
$$\int_M\psi\wedge\mathcal{D}^+_J(f)=0,$$
for any $f\in L_2^2(M)_0$.
\end{defi}
  \begin{rem}\label{0108}
 If $J$ is integrable, $\bar{\partial}^2_J=\partial^2_J=0$ and $\partial_J\bar{\partial}_J+\bar{\partial}_J\partial_J=0$,
  then $dJdf=2\sqrt{-1}\partial_J\bar{\partial}_Jf=\mathcal{D}^+_J(f)$. Hence, $\mathcal{D}^+_J$ can be viewed as a generalized
  $\partial\bar{\partial}$-operator in classical complex analysis. It is easy to see if
  the fundamental $2$-form $F$ is a $(1,1)$-component of a symplectic form on $M$, it is weakly $\mathcal{D}^+_J$-closed.
 $d^+_J(u+\bar{u})$ are also weakly $\mathcal{D}^+_J$-closed,
   where $u\in\Lambda^{0,1}_J\otimes L^2_1(M)$ \cite[Lemma 3.4]{TaWaZhZh}.
  \end{rem}
Suppose that $(M,g,J,F)$ is a closed almost Hermitian 4-manifold
 tamed by a symplectic form $F+d^-_J(v+\bar{v})$, for $v\in\Omega^{0,1}_J(M)$. Then we can define $\widetilde{\mathcal{D}}^+_J$ operator as follows \cite{TaWaZhZh}:
\begin{defi}\label{0104} Assume that $(M,g,J,F)$ is a closed almost Hermitian 4-manifold
 tamed by a symplectic form $F+d^-_J(v+\bar{v})$, where $v\in\Omega^{0,1}_J(M)$.
  Set $\mathcal{\widetilde{W}}: L^2_2(M)_0\longrightarrow\Lambda^1_\mathbb{R}\otimes L^2_1(M)$:
  $$\mathcal{\widetilde{W}}(f)=\mathcal{W}(f)-*_g(df\wedge d^-_J(v+\bar{v}))+d^*(\eta^2_f+\overline{\eta}^2_f),$$
where $\eta^2_f\in\Lambda^{0,2}_J\otimes L^2_2(M)$ satisfies
$$-d_J^-*_g(df\wedge d_J^-(v+\bar{v}))+d_J^-d^*(\eta^2_f+\bar{\eta}^2_f)=0.$$
Obviously,
   $$d^*\mathcal{\widetilde{W}}(f)=0, \,\,\, d^-_J\mathcal{\widetilde{W}}(f)=0.$$
Define $$\widetilde{\mathcal{D}}^+_J: L^2_2(M)_0\longrightarrow\Lambda^{1,1}_J\otimes L^2(M),$$
$$\widetilde{\mathcal{D}}^+_J(f)=d\mathcal{\widetilde{W}}(f).$$
 $\psi\in\Lambda^{1,1}_J\otimes L^2(M)$ is said to be weakly $\widetilde{\mathcal{D}}^+_J$-closed if
 $$\int_M\psi\wedge\widetilde{\mathcal{D}}^+_J(f)=0,$$
 for any $f\in L^2_2(M)_0$.
  \end{defi}
Obviously, $F$ and $d_J^+(u+\bar{u})$
 are  weakly $\widetilde{\mathcal{D}}^+_J$-closed, where $u\in\wedge_J^{0,1}\otimes L_1^2(M)$.
It is also worth remarking that if $dF=0$ (that is, $F$ is a symplectic form), then $dd_J^-(v+\bar{v})=0$ which implies that $d_J^-(v+\bar{v})=0$ by Stokes Theorem. Thus, $\widetilde{\mathcal{D}}^+_J=\mathcal{D}^+_J$. In the remainder of this paper, the fundamental $2$-form $F$ is said to be weakly $\widetilde{\mathcal{D}}^+_J$-closed if $F$ is the $(1,1)$-component of a symplectic form $\omega=F+d_J^-(v+\bar{v})$ for some $v\in \Omega_J^{0,1}(M)$.\\

  By using operator $\widetilde{\mathcal{D}}^+_J$ (resp. ${\mathcal{D}}^+_J$), we can define space of weakly $\widetilde{\mathcal{D}}^+_J$ (resp. ${\mathcal{D}}^+_J$)-closed $J$-$(1,1)$ forms as follows \cite{TaWaZhZh}:
 \begin{defi}\label{205}
 (i) Let $(M,g,J,F)$ be a closed almost Hermitian $4$-manifold. We define
 $$(\Lambda^{1,1}_J\otimes L^2(M))_{w}:=\{\psi\in\Lambda^{1,1}_J\otimes L^2(M)\,\,| \int_M\psi\wedge{\mathcal{D}}^+_J(f)=0, \forall f\in L^2_2(M)_0\}$$
which is  the space of weakly ${\mathcal{D}}^+_J$-closed $J$-(1,1) forms.
Furthermore, if $F$ is weakly ${\mathcal{D}}^+_J$-closed, then we define\\
$(\Lambda^{1,1}_J\otimes L^2(M))^0_{w}:=
  \{cF+\psi \mid \,\,\,c\in\mathbb{R},\,\,\,
   \psi\in(\Lambda^{1,1}_J\otimes L^2(M))_{w}$
  $$
   {\rm satisfies}\,\,\,
    P^+_g(\psi)\perp (H^+_g\cap H_J^+)\  {\rm and} \ F, \, \,\,{\rm with\,\,\, respect\,\,\, to\,\,\, the\,\,\, integration}\}.
  $$
(ii) Similarly, assume that $(M,g,J,F)$ is a closed almost Hermitian 4-manifold
 tamed by the symplectic form $F+d^-_J(v+\bar{v})$, for $v\in\Omega^{0,1}_J$. We define
 $$(\Lambda^{1,1}_J\otimes L^2(M))_{\widetilde{w}}:=\{\psi\in\Lambda^{1,1}_J\otimes L^2(M)\,\,| \int_M\psi\wedge{\widetilde{\mathcal{D}}}^+_J(f)=0, \forall f\in L^2_2(M)_0\}$$
which is  the space of weakly $\widetilde{{\mathcal{D}}}^+_J$-closed $J$-(1,1) forms.
we also define\\
$$(\Lambda^{1,1}_J\otimes L^2(M))^0_{\widetilde{w}}:=
  \{cF+\psi \mid \,\,\,c\in\mathbb{R},\,\,\,
   \psi\in(\Lambda^{1,1}_J\otimes L^2(M))_{\widetilde{w}}$$
  $$
   {\rm satisfies}\,\,\,
    P^+_g(\psi)\perp (H^+_g\cap H_J^+) \,\,\,{\rm with\,\,\, respect\,\,\, to\,\,\, the\,\,\, integration}\}.
  $$
  \end{defi}

  It is clear that
  $$(\Lambda^{1,1}_J\otimes L^2(M))^0_{w}\subset(\Lambda^{1,1}_J\otimes L^2(M))_{w}$$
  and
   $$(\Lambda^{1,1}_J\otimes L^2(M))^0_{\widetilde{w}}\subset(\Lambda^{1,1}_J\otimes L^2(M))_{\widetilde{w}}.$$\\

We would like to remark that for a tamed $4$-manifold, $\mathcal{H}_g^+(M)$ contains a symplectic form
$$\omega_1=F+d_J^-(v+\bar{v}),$$
 where $v\in\Omega_J^{0,1}(M)$. Let
$$\widetilde{\omega}_1=\omega_1-d(v+\bar{v})=F-d_J^+(v+\bar{v}).$$
Assume that
$$\int_MF^2=2.$$
Then
$$\int_M\widetilde{\omega}_1\wedge F=2(1+a),\ a\geq0.$$
If $a=0$, then $\widetilde{\omega}_1=\omega_1=F$.
Hence if $P_g^+(\psi)\perp (H^+_g\cap H_J^+)$ with respect to the integration, it follows that $P_g^+(\psi)\perp F$ with respect to the integration. In general,
for a closed almost Hermitian $4$-manifold $(M,g,J,F)$,
$P_g^+(\psi)\perp (H^+_g\cap H_J^+)$, it does not imply that $P_g^+(\psi)\perp F$ with respect to the integration.\\

 Similar to Lemma 3.2 in \cite{TaWaZhZh}, we have the following Lemma:
   \begin{lem}\label{0151} (i) If $(M,g,J,F)$ is a closed almost Hermitian $4$-manifold, then
  $\mathcal{D}^+_J$ has closed range.\\
  (ii)If  $(M,g,J,F)$ is a closed almost Hermitian $4$-manifold tamed by the symplectic form $\omega=F+d_J^-(v+\bar{v})$, where $v\in\Omega_J^{0,1}(M)$,
  then $\widetilde{\mathcal{D}}^+_J$ has closed range.
    \end{lem}
 \bp (Sketch) (i)
 Let $\{f_j\}$ be a sequence of real functions on $M$ in $L_2^2(M)_0$. By Definition \ref{0103},
  $$\mathcal{D}^+_J(f_i)=dJdf_i+dd^*(\eta_{f_i}^1+\bar{\eta}_{f_i}^1)$$
   $$d_J^-Jdf_i+d_J^-d^*(\eta_{f_i}^1+\bar{\eta}_{f_i}^1)=0,$$
   where $\eta_{f_i}^1\in\wedge_J^{0,1}\otimes L_2^2(M)$. Hence,
    $$\mathcal{D}^+_J(f_i)=d_J^+Jdf_i+d_J^+d^*(\eta_{f_i}^1+\bar{\eta}_{f_i}^1)
    =2\sqrt{-1}\partial_J\bar{\partial}_Jf_i+d_J^+d^*(\eta_{f_i}^1+\bar{\eta}_{f_i}^1).$$
   By Lemma \ref{2130} and regularity of linear elliptic operator \cite{Au},
    \begin{eqnarray*}
    \int_M|\mathcal{D}^+_Jf_i|dvol_g&\geq&C_1\int_M|\partial_J\bar{\partial}_Jf_i|^2dvol_g-C_2\int_M|\nabla^2((\eta_{f_i}^1+\bar{\eta}_{f_i}^1))|^2dvol_g\\
    &\geq&C_1\|f_i\|_{L_2^2}^2-C_2\|\nabla^2 f_i\|_{L^2}^2.
     \end{eqnarray*}
  By Interpolation Inequalities \cite{Au}, it follows that
    \begin{eqnarray*}
  \|\mathcal{D}_J^+f_i\|_{L2}^2\geq C_3\|f_i\|_{L_2^2}^2-C_4,
     \end{eqnarray*}
   where $C_1$, $C_2$, $C_3$ and $C_4$ are all positive constants depending only on $(M,g)$.
   Since $\mathcal{D}_J^+(f_i)$ converges to $\psi\in \wedge_J^+\otimes L^2(M)$ in $L^2$ space, it follows that
   $f_i$ are bounded in $L_2^2$.

   As done in the proof of Lemma \ref{311}, there is a subsequence of $\{f_i\}$ converges (resp. weakly) in $L_1^2$ (resp. $L_2^2$) to some
   $f\in L_2^2(M)_0$ such that $\mathcal{D}_J^+(f)=\psi$. Thus, $\mathcal{D}_J^+$ has closed range.

(ii) Suppose that $\widetilde{\mathcal{D}}_J^+(f_j)$ converges to $\psi\in \wedge_J^{1,1}\otimes L^2(M)$ for $\{f_j\}\subset L_2^2(M)_0$.
By Definition \ref{0104},
$$\widetilde{\mathcal{D}}^+_J(f_j)=2\sqrt{-1}\partial_J\bar{\partial}_Jf_j-d_J^+*_g(df\wedge(v+\bar{v}))+d_J^+d^*(\eta_f+\bar{\eta}_f),$$
where
 $$v\in\Omega_J^{0,1}(M), \eta_f\in\wedge_J^{0,2}\otimes L_2^2(M), \omega_1=F+d_J^-(v+\bar{v})$$
 is the tamed symplectic form and $v, \eta_f$
satisfy the following equation:
$$-d_J^-*_g(df\wedge d_J^-(v+\bar{v}))+d_J^-d^*(\eta_f+\bar{\eta}_f)=0.$$
By using similar method in the proof of (i), it is easy to see that there is
$f\in L_2^2(M)_0$ such that $\psi=\widetilde{\mathcal{D}}^+_J(f)$. Thus,  $\widetilde{\mathcal{D}}_J^+$ has closed range.
 \ep

\section{The intersection pairing on weakly $\widetilde{\mathcal{D}}_J^+$ (resp. $\mathcal{D}_J^+$)-closed $J$-(1,1) forms}
 \setcounter{equation}{0}
 In this section, we consider the intersection pairing on weakly $\widetilde{\mathcal{D}}_J^+$ (resp. $\mathcal{D}_J^+$)-closed $J$-(1,1) forms on a
 closed almost Hermitian $4$-manifold $(M,g,J,F)$.
 As done in compact complex surfaces \cite{BaPeVa}, there is Signature Theorem (cf. \cite[Theorem2.4]{TaWaZhZh})
 for closed almost complex $4$-manifolds as follows:
 \begin{prop}\label{01}
 Let $(M,J)$ be a closed almost complex $4$-manifold. Then the cup-product form on $H^2(M,\textrm{R})$ restricted to $H^+_J$
 is nondegenerate of type $(b^+-h^-_J,b^{-})$.
 \end{prop}
 \bp
 It is omitted. See the proof of Theorem 2.4 in \cite{TaWaZhZh}.
 \ep
We now investigate the intersection pairing
 on weakly $\widetilde{\mathcal{D}}_J^+$ (resp. $\mathcal{D}_J^+$)-closed $J$-(1,1) forms (c.f. \cite{Bu,TaWaZhZh,Zh}).

Under some condition, we can solve $(\mathcal{\widetilde{W}},d_J^-)$-problem (or $(\mathcal{W},d_J^-)$-problem)
as solving $\bar{\partial}$-problem in classical complex analysis \cite{Ho0,Ho}. Hence we have the following proposition
(cf. \cite[Lemma 3.7]{TaWaZhZh}):
\begin{prop}\label{0110}
Suppose that  $(M,g,J,F)$ is a closed almost Hermitian $4$-manifold.
  If
  $$\psi\in(\Lambda^{1,1}_J\otimes L^2(M))^0_{\widetilde{w}}
  \  ({\rm or}\  \psi\in(\Lambda^{1,1}_J\otimes L^2(M))^0_{w} ),$$
   then $\psi$ could be written as
  $$\psi=C_\psi F+\beta_\psi+\widetilde{\mathcal{D}}^+_J(f_\psi)+d^-_g(\gamma_\psi), $$
 $$\  ({\rm or}\ \psi=C_\psi F+\beta_\psi+{\mathcal{D}}^+_J(f_\psi)+d^-_g(\gamma_\psi)),$$
  where $C_\psi$ is a constant,
  $$\beta_{\psi}\in\mathcal{H}^-_g(M), f_\psi\in L^2_2(M)_0,\gamma_\psi\in\Lambda^1_\mathbb{R}\otimes L^2_1(M)$$ and
    $$d^-_g(\gamma_\psi)\in(\Lambda^{1,1}_J\otimes L^2(M))^0_{\widetilde{w}}\  ({\rm or}\  d^-_g(\gamma_\psi)\in(\Lambda^{1,1}_J\otimes L^2(M))^0_{w}).$$
 Let $\psi'=\psi-d^-_g(\gamma_\psi)$. Then
   $$
     \int_M\psi^2=\int_M\psi'^2 -\|d^-_g(\gamma_\psi)\|^2_{L^2(M)},
  $$
  and there is a smooth sequence of $\{f_{\psi,j}\}\subset C^\infty(M)_0$, where
  $$C^{\infty}(M)_0:=\{f\in C^{\infty}(M):\int_Mf d\mu_g=0\}$$
  such that
  $$\psi'_j=C_\psi F+\beta_\psi+\widetilde{\mathcal{D}}^+_J(f_{\psi,j})\  ({\rm or} \ \psi'_j=C_\psi F+\beta_\psi+{\mathcal{D}}^+_J(f_{\psi,j}))$$
  is converging to $\psi'$ in $L^2(M)$.
  \end{prop}
  \bp We may assume that $$\int_MF^2=2.$$
Set
  $$C_{\psi}:=\frac{1}{2}\int_M\psi\wedge F. $$
Note that
  $$\Lambda^+_g=\mathbb{R}\cdot F\oplus \Lambda^-_J,\,\,\, \Lambda^+_J=\mathbb{R}\cdot F\oplus \Lambda^-_g,$$
  and $F$ is weakly $\widetilde{\mathcal{D}}^+_J$-closed. Since $\psi\in(\wedge_J^{1,1}\otimes L^2(M))_{\widetilde{w}}^0$, obviously,  $P^+_g(\psi-C_{\psi}F)$
  is orthogonal to $\mathcal{H}^+_g \cap H^+_J$ with respect to the integration. Hence $\psi-C_{\psi}F$ is orthogonal to $F$ with respect to the integration,
  and $(\psi-C_{\psi}F)\perp\Omega_J^-(M)$.
  By Hodge decomposition  \cite{DoKr} $$P_g^+(\psi-C_{\psi}F)=d_g^+\eta_\psi$$ for some $\eta_\psi\in\Omega_{\mathbb{R}}^1(M)$, and
  $d_J^-\eta_\psi=0$. By  Lemma  \ref{65},
  there exists $f_{\psi}\in L^2_2(M)_0$
  such that
  \begin{equation}\label{3eq15}
   P^+_g(\psi-C_{\psi}F)=P^+_g(\widetilde{\mathcal{D}}_J^+(f_{\psi}))
  \end{equation}
  holds in $\Lambda^{1,1}_J\otimes L^2(M)$.
  If $\psi$ is smooth, then $f_{\psi}$ is also smooth.
    By (\ref{3eq15}), we can find that
  $$
   \psi-C_{\psi}F-\widetilde{\mathcal{D}}^+_J(f_{\psi})=P^-_g(\psi-C_{\psi}F-\widetilde{\mathcal{D}}^+_J(f_{\psi}))\in\Lambda^-_g\otimes L^2(M),
   $$
  since $$P^+_g(\psi-C_{\psi}F-\widetilde{\mathcal{D}}^+_J(f_{\psi}))=0.$$
  By Hodge decomposition $$\Omega_g^-(M)=\mathcal{H}_g^-\oplus d_g^-(\Omega^1(M))$$ again, we obtain the following equation:
  $$
   \psi-C_{\psi}F-\widetilde{\mathcal{D}}^+_J(f_{\psi})=\beta_{\psi}+d^-_g(\gamma_\psi),
   $$
  where $$\beta_{\psi}\in\mathcal{H}^-_g(M),\  \gamma_\psi\in\Lambda^1_\mathbb{R}\otimes L^2_1(M).$$
   Hence,
  $$
  \psi=C_{\psi}F+\beta_{\psi}+d^-_g(\gamma_\psi)+\widetilde{\mathcal{D}}^+_J(f_{\psi}).
   $$
Since $$\psi,\  F, \ \beta_{\psi},\  \widetilde{\mathcal{D}}^+_J(f_{\psi})\in(\Lambda^{1,1}_J\otimes L^2(M))^0_{\widetilde{w}},$$
 we have
 $$d^-_g(\gamma_\psi)\in(\Lambda^{1,1}_J\otimes L^2(M))^0_{\widetilde{w}}$$
  is weakly $\widetilde{\mathcal{D}}_J^+$ closed.
  Set $$\psi':=\psi-d^-_g(\gamma_\psi)=C_{\psi}F+\beta_{\psi}+\widetilde{\mathcal{D}}^+_J(f_{\psi}).$$
  $\psi'$ is also in $(\Lambda^{1,1}_J\otimes L^2(M))^0_{\widetilde{w}}$.
  If $\psi$ is smooth,  both $\psi'$ and $f_{\psi}$ are smooth.
    Then, we have the following equation:
  \begin{equation}\label{smooth case}
  F\wedge(\psi'-C_{\psi}F-\widetilde{\mathcal{D}}^+_J(f_{\psi}))=F\wedge\beta_\psi=0.
  \end{equation}
   If $\psi$ is not smooth, in $\Lambda^{1,1}_J\otimes L^2(M)$,
   we still have
   $$
  \psi=C_{\psi}F+\beta_{\psi}+d^-_g(\gamma_\psi)+\widetilde{\mathcal{D}}^+_J(f_{\psi}),
   $$
   and
  $$ \psi'=C_{\psi}F+\beta_{\psi}+\widetilde{\mathcal{D}}^+_J(f_{\psi}),$$
  in the sense of currents.
     Since $d^-_g(\gamma_\psi)\in(\Lambda^{1,1}_J\otimes L^2(M))_{\widetilde{w}}^0$, it is easy to see that
     $$
     \int_M\psi'\wedge d^-_g(\gamma_\psi)=\int_M(C_\psi F+\beta_\psi+\widetilde{\mathcal{D}}^+_J(f_\psi))\wedge d^-_g(\gamma_\psi)=0,
     $$
     and
  \begin{eqnarray*}
     \int_M\psi^2=\int_M\psi'^2 -\|d^-_g(\gamma_\psi)\|^2_{L^2(M)}.
  \end{eqnarray*}
  By Lemma \ref{0151}, we can find a smooth sequence of $\{f_{\psi,j}\}\subset C^\infty(M)_0$ such that
  $\widetilde{\mathcal{D}}^+_J(f_{\psi,j})$ is converging to  $\widetilde{\mathcal{D}}^+_J(f_{\psi})$ in $L^2(M)$. Hence,
  this implies that
  $$\psi'_j=C_{\psi}F+\beta_{\psi}+\widetilde{\mathcal{D}}^+_J(f_{\psi,j})$$
  is converging to $\psi'$ in $L^2(M)$.

  If $\psi\in(\Lambda^{1,1}_J\otimes L^2(M))^0_{w}$, by the same method, we have the similar result.
  
  This completes the proof of Proposition \ref{0110}.
  \ep

 By Proposition \ref{0110}, we have the following lemma (c.f. \cite[Lemma 4]{Bu}, \cite[Lemma 3.8]{TaWaZhZh})
 which will be used to prove Proposition \ref{0112}:
   \begin{lem}\label{0111}
Suppose that  $(M,g,J,F)$ is a closed almost Hermitian $4$-manifold. If
$$\psi\in(\Lambda^{1,1}_J\otimes L^2(M))^0_{\widetilde{w}}\  (\ {or}\  \psi\in(\Lambda^{1,1}_J\otimes L^2(M))^0_{{w}}),$$
   then
  $$ (\int_MF\wedge\psi)^2\geq(\int_MF^2)(\int_M\psi^2) $$
  with equality if and only if
  $$\psi=C_{\psi}F+\widetilde{\mathcal{D}}^+_J(f)\ \ \ ({\it or}\ \ \  \psi=C_{\psi}F+{\mathcal{D}}^+_J(f))$$
   for some constant
   $$C_{\psi}=\frac{1}{2}\int_MF\wedge\psi$$
   and  $f\in L^2_2(M)_0$.
  \end{lem}
\bp
See the proof of Lemma 3.8 in \cite{TaWaZhZh} which is the light cone lemma \cite[Lemma 2.6]{LiLi}.
\ep
By Lemma \ref{0111}, we have the following proposition (c.f. \cite[Proposition 5]{Bu},\cite[Proposition 3.9]{TaWaZhZh}):
   \begin{prop}\label{0112}
 Suppose that  $(M,g,J,F)$ is a closed almost Hermitian $4$-manifold. Let
 $$\psi_1,\,\,\,\psi_2\in(\Lambda^{1,1}_J\otimes L^2(M))^0_{\widetilde{w}}
 \ ({\rm or} \  \psi_1,\,\,\,\psi_2\in(\Lambda^{1,1}_J\otimes L^2(M))^0_{{w}})$$
   and satisfy
   $$\int_M\psi^2_j\geq 0 \,\,\,and\,\,\, \int_MF\wedge\psi_j\geq 0$$ for $j=1,2$.
  Then
  $$
  \int_M\psi_1\wedge\psi_2\geq(\int_M\psi_1^2)^{\frac{1}{2}}(\int_M\psi_2^2)^{\frac{1}{2}},
  $$
  with equality if and only if $\psi_1$ and $\psi_2$ are linearly dependent modulo the image of $\widetilde{\mathcal{D}}^+_J$ (or ${\mathcal{D}}^+_J$).
  \end{prop}
\bp
See the proof of Proposition 3.9 in \cite{TaWaZhZh}.
\ep
By Proposition \ref{0112}, it is easy to obtain the following corollary (c.f. \cite[Corollary 6]{Bu},\cite[Corollary 3.10]{TaWaZhZh}):
   \begin{col}\label{0113}
Suppose that  $(M,g,J,F)$ is a closed almost Hermitian $4$-manifold. If
 $$\psi\in(\Lambda^{1,1}_J\otimes L^2(M))^0_{\widetilde{w}}\ (\ or \ \psi\in(\Lambda^{1,1}_J\otimes L^2(M))^0_{{w}})$$
  satisfies
 $$\int_M\psi^2>0 \,\,\,and \,\,\,\int_M\psi\wedge F>0,$$
 then $$\int_M\psi\wedge\varphi>0$$ for any other such form
 $$\varphi\in(\Lambda^{1,1}_J\otimes L^2(M))^0_{\widetilde{w}}\  ({\rm{or}} \ \varphi\in(\Lambda^{1,1}_J\otimes L^2(M))^0_{w})$$
  satisfying $$\int_M\varphi^2\geq 0\,\,\, and \,\,\,\int_M\varphi\wedge F>0.$$
 \end{col}
By Corollary \ref{0113}, Lemma \ref{65} and Hahn-Banach Theorem \cite{Yo}, using the method developed by Sullivan \cite{Su},
we have the following key lemma (c.f. \cite[Lemma 7]{Bu}, \cite[Lemma 3.12]{TaWaZhZh}):
  \begin{lem}\label{0115}
  (1) Let $(M,J)$ be a closed tamed almost complex $4$-manifold. Suppsose that
  $\varphi\in (\wedge_J^{1,1}\otimes L^2(M))_{\widetilde{w}}^0$ and satisfies
  $$\int_M\varphi\wedge F\geq0$$
  and
  $$\int_M\varphi^2\geq0.$$
  For each $\varepsilon>0$, there is a positive $J$-$(1,1)$ form $p_\varepsilon$ and a function $f_\varepsilon$
  such that
  $$\|\varphi+\widetilde{\mathcal{D}}^+_J(f_\varepsilon)-p_\varepsilon\|_{L^2(M)}<\varepsilon.$$
  Moreover, $p_\varepsilon$ and $f_\varepsilon$ can be assumed to be smooth.

  (2) Suppose that $(M,g,J,F)$ is a closed almost Hermitian 4-manifold,
  where $F$ is weakly  ${\mathcal{D}}^+_J$-closed.
  Suppose that
  $$\varphi\in(\Lambda^{1,1}_J\otimes L^2(M))_{w}^0$$
  and satisfies
  $$
   \int_M\varphi\wedge F\geq 0\,\,\,{\rm and} \,\,\,\int_M\varphi^2\geq 0.
   $$
   For each $\varepsilon>0$ there is a positive $(1,1)$-form $p_\varepsilon$ and a function $f_\varepsilon$
   such that
  $$\|\varphi+{\mathcal{D}}^+_J(f_\varepsilon)-p_\varepsilon\|_{L^2(M)}<\varepsilon.
  $$
   Moreover, $p_\varepsilon$ and $f_\varepsilon$ can be assumed to be smooth.
   \end{lem}
\bp (Sketch) Let $\omega=F+d_J^-(v+\bar{v})$ be the $J$-taming symplectic form, where $v\in\Omega_J^{0,1}(M)$.
Then $\widetilde{\omega}=F-d_J^+(v+\bar{v})\in H_g^+\cap H_J^+$, and
$$\int_M\widetilde{\omega}\wedge F>0.$$
  (1) Suppose that $\varphi\in(\Lambda^{1,1}_J\otimes L^2(M))^0_{\widetilde{w}}$.
   If $$\int_M\varphi\wedge F=0,$$ by Proposition \ref{0110}, it follows that
  \begin{equation}
     \varphi=\beta_{\varphi}+\widetilde{\mathcal{D}}^+_J(f_{\varphi})+d^-_g(\gamma_{\varphi}).
  \end{equation}
   Then
  $$
  0\geq -\|\beta_{\varphi}\|^2_{L^2(M)}-\|d^-_g(\gamma_{\varphi})\|^2_{L^2(M)}
    =\int_M(\varphi-\widetilde{\mathcal{D}}^+_J(f_{\varphi}))^2=\int_M\varphi^2\geq 0.
  $$
  This implies that $\beta_\varphi=0$ and $d_g^-(\gamma_\varphi)=0$.
  Thus, we  get $\varphi=\widetilde{\mathcal{D}}^+_J(f_{\varphi})$, that is,
  $\varphi$ is $\widetilde{\mathcal{D}}^+_J$ exact.
  In this case the result follows from the denseness of the smooth functions in $L^2_2(M)_0$.\\

Now we may assume without loss of generality, that $$\int_M\varphi\wedge F>0.$$
  After rescaling $\varphi$ if necessary, it can be supposed that
  $$\int_M\varphi\wedge F=1.$$
  Let
   \begin{equation}\label{3eq20}
  \mathcal{P}:=\{p\in\Lambda^{1,1}_J\otimes L^2(M)\setminus(H^+_g\cap H_J^+)+\mathbb{R}F\mid  p\geq 0,\,\,\, a.e.,\,\,\, \int_Mp\wedge F=1\};
   \end{equation}
  \begin{equation}\label{3eq21}
  \mathcal{P}_\varepsilon:=\{\rho\in\Lambda^{1,1}_J\otimes L^2(M)\setminus(H^+_g\cap H_J^+)+\mathbb{R}F\mid  \|\rho-p\|_{L^2(M)}<\varepsilon \,\,\, for\,\,\, some \,\,\,p\in\mathcal{P}\};
  \end{equation}
  \begin{equation}\label{3eq22}
  \mathcal{H}_\varphi:=\{\varphi+\widetilde{\mathcal{D}}^+_J(f)\mid  f\in L^2(M)_0\}.
  \end{equation}
  Then $\mathcal{P}_\varepsilon$ is an open convex subset of the Hilbert space
  $$H:=\Lambda^{1,1}_J\otimes L^2(M)\setminus(H^+_g\cap H_J^+)+\mathbb{R}F,$$
   and $\mathcal{H}_\varphi$ is a closed convex subset since $\widetilde{\mathcal{D}}^+_J$ has closed range by Lemma \ref{0151}.
   It is easy to see that
   $$H\cap (\wedge_J^{1,1}\otimes L^2(M))_{\widetilde{\omega}}=(\wedge_J^{1,1}\otimes L^2(M))^0_{\widetilde{\omega}}.$$
  Notice that $0\leq h_J^-\leq b^+-1$ \cite{TaWaZh}. If $\mathcal{P}_\varepsilon\cap\mathcal{H}_\varphi=\emptyset$,
  Hahn-Banach Theorem \cite{Yo} implies that there exists $\phi\in H$ and a constant $c\in\mathbb{R}$ such that
  \begin{equation}\label{3eq23}
  \int_M\phi\wedge h\leq c, \,\,\,\int_M\phi\wedge p>c,
  \end{equation}
  for every $h\in\mathcal{H}_\varphi$,
   and every $p\in\mathcal{P}_\varepsilon$ (Compare the proof of Theorem I.7 in D. Sullivan \cite{Su} and the proof of Lemma 7 in N. Buchdahl \cite{Bu}).
   In terms of (\ref{3eq22}) and (\ref{3eq23}), there exists a $f_{\phi}\in L^2_2(M)_0$ such that $h_{\phi}=\varphi+\widetilde{\mathcal{D}}^+_J(f_{\phi})$ and
  $$\int_M\phi\wedge h_{\phi}=c,$$
  since $\mathcal{H}_\varphi$ is a closed space.
For any $h\in\mathcal{H}_\varphi$, it follows that $h-h_{\phi}$ is in the image of $\widetilde{\mathcal{D}}^+_J$.
  Hence,
  \begin{equation*}
  \int_M\phi\wedge(h-h_{\phi})\leq 0,\,\,\,\int_M\phi\wedge(h_{\phi}-h)\geq 0.
   \end{equation*}
  It follows immediately that $\phi$ is weakly $\widetilde{\mathcal{D}}^+_J$-closed since $h-h_{\phi}$ and $h_{\phi}-h$ are in the image of $\widetilde{\mathcal{D}}^+_J$, that is,
  $$\int_M\phi\wedge\widetilde{\mathcal{D}}^+_J(f)=0$$
  for any $f\in L^2_2(M)_0$.
 Hence $\phi\in(\Lambda^{1,1}_J\otimes L^2(M))^0_{\widetilde{w}}$ since
 $$\phi\in H\cap(\wedge_J^{1,1}\otimes L^2(M))_{\widetilde{\omega}}.$$

  Let $$\phi_0:=\phi-cF\in(\Lambda^{1,1}_J\otimes L^2(M))^0_{\widetilde{w}}.$$
 By (\ref{3eq20}) and (\ref{3eq23}), we have
  \begin{equation*}
  \int_M\phi_0\wedge\varphi\leq c-c=0,
   \end{equation*}
  and
  \begin{equation*}
  \int_M\phi_0\wedge p_0>0
  \end{equation*}
   for any $p_0\in\mathcal{P}$.
  So $\phi_0$ is strictly positive almost everywhere.
  Hence
  $$\int_M\phi_0^2>0\,\,\, and \,\,\, \int_M\phi_0\wedge F>0.$$
  It follows from Corollary \ref{0113} that
  $$\int_M\phi_0\wedge \varphi>0,$$
   giving a contradiction.
  Therefore $\mathcal{P}_\varepsilon\cap\mathcal{H}_\varphi$ can not be empty proving the existence of $p_\varepsilon$ and $f_\varepsilon$.
   The last statement of the part (1) of the lemma follows from denseness of the smooth positive $(1,1)$-forms
    in the $L^2$-positive $(1,1)$-forms and of the smooth functions in $L^2_2(M)_0$.

(2) Since $F$ is weakly $\mathcal{D}_J^+$-closed,
we consider $(\Lambda^{1,1}_J\otimes L^2(M))^0_{w}$ case.
For $\varphi\in (\Lambda^{1,1}_J\otimes L^2(M))^0_{w}$, without loss of generality, by rescaling $\varphi$, we may assume that
$$\int_M\varphi\wedge F=1.$$
Since $0\leq h_J^-\leq b^+$ \cite{TaWaZh}, let
$$H:=\wedge_J^{1,1}\otimes L^2(M)\setminus (H_g^+\cap H_J^+)+\mathbb{R}F$$
be an open subspace of the Hilbert space $\wedge_J^{1,1}\otimes L^2(M)$. Similar to the proof of Part (1), define
\begin{align}\label{312}
\mathcal{P}:=\{p\in H|p\geq0 \ {\rm a.e.,}\ \int_Mp\wedge F=1\};
\end{align}
\begin{align}\label{313}
\mathcal{P_\varepsilon}:=\{\rho\in H| \|\rho-p\|_{L^2(M)}<\varepsilon\ {\rm for \ some\ }p\in \mathcal{P}\};
\end{align}
\begin{align}\label{314}
\mathcal{H}_\varphi:=\{\varphi+\mathcal{D}_J^+(f)| f\in L_2^2(M)_0\}.
\end{align}
Then $\mathcal{P_\varepsilon}$ is an open convex subset of $H$ and $\mathcal{H}_\varphi$ is a closed convex subset since
$\mathcal{D}_J^+$ has closed range by Lemma \ref{0151}.

We should prove that $\mathcal{H}_\varphi\cap\mathcal{P}_\varepsilon\neq\emptyset$ for any $\varepsilon>0$. If
$\mathcal{H}_\varphi\cap\mathcal{P}_\varepsilon=\emptyset$, Hahn-Banach Theorem \cite{Yo} implies that there exists $\phi\in H$
and a constant $c\in \mathbb{R}$ such that
\begin{align}\label{315}
\int_M\phi\wedge h\leq c, \int_M\phi\wedge p> c,
\end{align}
for every $h\in\mathcal{H}_\varphi$ and every $p\in\mathcal{P_\varepsilon}$. As
simiar argument in the proof of Part (1), it follows that
$$\int_M\phi\wedge \mathcal{D}_J^+(f)=0$$
for any $f\in L_2^2(M)_0$. Hence $\phi\in(\wedge_J^{1,1}\otimes L^2(M))_w$,
 since $\phi\in H$ it follows that $\phi\in(\wedge_J^{1,1}\otimes L^2(M))_w^0$. Let
$$\phi_0:=\phi-cF\in(\wedge_J^{1,1}\otimes L^2(M))_w^0.$$
By \eqref{312} and \eqref{315}, we have
$$\int_M\phi_0\wedge\varphi\leq c-c=0,$$
and
$$\int_M\phi_0\wedge p_0>0$$
for any $p_0\in \mathcal{P}$. So $\phi_0$ is strictly positive almost everywhere. Hence
$$\int_M\phi_0^2>0$$
and
$$\int_M\phi_0\wedge F>0.$$
It follows from Corollary  \ref{0113} that
$$\int_M\phi_0\wedge\varphi>0,$$
giving a contradiction. Therefore, $\mathcal{H}_\varphi\cap\mathcal{P}_\varepsilon\neq\emptyset$.

This completes the proof of Lemma \ref{0115}.
\ep

\section{A characterization of tamed $4$-manifolds}
 \setcounter{equation}{0}
This section is devoted to giving a characterization of those closed almost complex $4$-manifolds which admit tamed or
weakened tamed symplectic structures.\\

It is worth remarking that  if $(M,J_0)$ is a compact complex surface, then $h_{J_0}^-=b^+-1$ if and only if $b^1(M)$ is even, and $h_{J_0}^-=b^+$ if and
only if $b^1(M)$ is odd \cite{BaPeVa,TaWaZhZh}.
One can prove that if $b^1(M)$ is even, then $(M,J_0)$
admits a K\"{a}hler form \cite[Theorem 11]{Bu}. For a closed tamed $4$-manifold $(M,J)$, Tan, Wang, Zhou and Zhu proved
 the following tameness result \cite[Theorem 1.1]{TaWaZhZh}:
 \begin{theo}\label{th3}
Suppose that $(M,g,J,F)$ is a tamed closed almost Hermitian 4-manifold with $h_J^-=b^+-1$, where $J$-taming sympectic form is
 $$\omega=F+d_J^-(v+\bar{v}), v\in \Omega_J^{0,1}.$$
Then there exists a $J$-compatible symplectic form on $M$.
 \end{theo}
{\it Proof (Sketch). }
It is clear that $F$ is weakly $\widetilde{\mathcal{D}}^+_J$-closed. Note that
$$H^+_J(M)=(H_g^+(M)\cap H_J^+(M))\oplus \mathcal{H}_g^-(M).$$
Since  $h^-_J=b^+-1$, the dimension of $H^+_J(M)$ is $b^-+1$,
the dimension of $H_g^+(M)\cap H_J^+(M)$ is $1$. We normalized such that
 $$\int_MF^2=2.$$
Set
$$\widetilde{\omega}:=\omega-d(v+\bar{v})=F-d_J^+(v+\bar{v}).$$
Since $h_J^-=b^+-1$,
by Lemma \ref{67}, $\widetilde{\omega}\in (\wedge_J^{1,1}\otimes L^2(M))_{\widetilde{w}}^0$.
Obviously, $\widetilde{\omega}$ is $d$-closed and belongs to $\Omega_J^+.$
Let
   \begin{equation}\label{0161}
\int_Md_J^-(v+\bar{v})\wedge d_J^-(v+\bar{v})=2a\geq0.
  \end{equation}
  If $a=0$, then $F$ is a $J$-compatible symplectic form. Now we suppose that $a>0$. Obviously,
   \begin{equation}\label{0162}
\int_M\widetilde{\omega}^2=\int_M\omega^2=2(1+a),
  \end{equation}
     \begin{equation}\label{0163}
\int_Md_J^+(v+\bar{v})\wedge d_J^+(v+\bar{v})=-2a,
  \end{equation}
  and
     \begin{equation}\label{0164}
\int_MF\wedge d_J^+(v+\bar{v})=-2a.
  \end{equation}
Let
$$t_0:=\left(1+\sqrt{\frac{a}{1+a}}\right)^{-1}.$$
 It is easy to see that $t_0\in(0,1)$.
By Proposition \ref{0110} and \eqref{0162}-\eqref{0164},
the smooth $J$-$(1,1)$-form
$$\varphi=\widetilde{\omega}-t_0F=(\sqrt{a(1+a)}-a)F-d_J^+(v+\bar{v})$$
is still in $(\wedge_J^{1,1}\otimes L^2(M))_{\widetilde{w}}^0$. It is easy to see that,
by Lemma \ref{0115},
for each $n\in \mathbb{N}$, there is a smooth positive $J$-$(1,1)$-form $p_n$ and a smooth function $g_n$ such that
 $$\|\varphi+\widetilde{\mathcal{D}}^+_J(g_n)-p_n\|_{L^2(M)}<\frac{1}{n}. $$
 Since
 $$\int_Mp_n\wedge F$$
 converges to
  $$\int_M\varphi\wedge F=2\sqrt{a(1+a)}>0,$$
 the positive functions $\wedge p_n$ are
 uniformly bounded in $L^2$, so a subsequence can be found converging weakly in $L^2$. The forms $\frac{p_n}{\wedge p_n}$
 are bounded in $L^{\infty}$, so subsequence of these forms can also be found in $L^1$ converging weakly in $L^4$. The sequence
$\{\widetilde{\mathcal{D}}^+_J(g_n)\}=\{d\mathcal{\widetilde{W}}(g_n)\}$ is uniformly bounded in $L^1$. The uniform $L^1$ bound
on  $\{\widetilde{\mathcal{D}}^+_J(g_n)\}$ does not imply an $L^2$ bound on $g_n$. We need to find a subsequence converge
in the sense of currents \cite{HaLa,TaWaZhZh}.
Here, we need the following lemma \cite[Claim 4.1]{TaWaZhZh}:

  \begin{lem}\label{0120}
 Given any $s<\frac{4}{3}$ and $t<2$, there is a subsequence of $\{g_n\}$ that converges weakly in $L^s_1$, and
 strongly in $L^t$ to a limiting function $g_0$.
 \end{lem}


We return to prove Theorem \ref{th3}.  By Lemma \ref{0120}, the subsequence of positive $J$-$(1,1)$-forms $\{p_n\}$ in the sense of currents defines a positive
 $J$-$(1,1)$- current
\begin{equation*}
p:=\varphi+\widetilde{\mathcal{D}}^+_J(g_0)=\widetilde{\omega}-t_0F+\widetilde{\mathcal{D}}^+_J(g_0),
 \end{equation*}
  $g_0\in L^q_2(M)_0$ for some fixed $q\in (1,2)$.
Since $\wedge p\in L^1$ and
$$p/(\wedge p)\in\Lambda^{1,1}_J\otimes L^\infty,$$ the $J$-almost K\"{a}hler current
\begin{equation}\label{0131}
P=\widetilde{\omega}+\widetilde{\mathcal{D}}^+_J(g_0)=p+t_0F\geq t_0F,
\end{equation}
in the sense of currents. More details see the proof of Proposition 4.2 in \cite{TaWaZhZh}.
 By \eqref{0131}, it is easy to see that
 \begin{equation}\label{0135}
\int_DP>0,
 \end{equation}
for each irreducible $J$-holomorphic curve (in general, subvariety \cite{Ta}) $D$ \cite{El,El1,LiZh2} on $M$ with $D^2\geq0$. To complete the proof of Theorem \ref{th3}, we will construct a $J$-almost K\"{a}hler form by using the $J$-almost
K\"{a}hler current.\\

Let $v(P,x)$ denote the Lelong number of $P$ at $x$ (c.f. Definition B.13 in \cite{TaWaZhZh} or \cite{El1,Si}). For $c>0$, the upperlevel
set
 \begin{equation}\label{0137}
E_c(P):=\{x\in M|v(P,x)\geq c\}
 \end{equation}
is a $J$-analytic subset of $M$ of dimension (complex) $\leq1$ by Siu's decomposition theorem \cite{El,El1,Si} for closed tamed $4$-manifolds (Appendix B.1 in \cite{TaWaZhZh}).
If $D$ is an irreducible $J$-holomorphic curve \cite{LiZh2,McSa0} in $E_c(P)$,
$$v_0:=\inf\{v(P,x)|x\in D\},$$
$v(P,x)=v_0$ for almost all $x\in D$  \cite[Lemma B.9]{TaWaZhZh}.
 If $D_1,\cdots,D_m$ are the irreducible $J$-holomorphic curves in $E_c(P)$ and
$$v_i:=\inf\{v(P,x)|x\in D_i\},$$
 the $d$-closed $J$-$(1,1)$-current
 \begin{equation}\label{0138}
 T=P-\Sigma v_iT_{D_i}
  \end{equation}
 is positive and the $c$-upper level set $E_c(T)$ of $T$ are isolated singular points (B.21 in \cite{TaWaZhZh}). Here
 $T_{D_i}$ are $J$-$(1,1)$ currents of integration on $D_i$ \cite{GrHa,LiZh2}, for $1\leq i\leq m$. As done in \cite{Bu}, using the method of
 Demailly \cite{De}, it is always possible to approximate the closed positive $J$-(1,1) current $T$ on a $J$-almost K\"{a}hler $4$-manifold
 by smooth real currents admitting a small negative part and that this negative part can be estimated in terms of the Lelong numbers
 of $T$ and geometry of $J$-almost K\"{a}hler $4$-manifold \cite[Theorem C.12, Remark C.13]{TaWaZhZh}. Since any $J$-almost complex $4$-manifold has the local
 symplectic property \cite{Le,Sik}, there is a $1$-parameter family $T_{c,\epsilon}$, of $d$-closed positive $J$-$(1,1)$-currents in the same homology class
 as $ T=P-\Sigma v_iT_{D_i}$ in the sense of currents which weakly converges to $T$
 as $\epsilon\rightarrow 0^+$, with $T_{c,\epsilon}$ smooth off $E_c(T)$,
 \begin{equation}\label{0139}
T_{c,\epsilon}\geq(t_0-\min\{\lambda_{\epsilon},c\}K-\delta_{\epsilon})F,
  \end{equation}
for some continuous functions $\lambda_{\epsilon}$ on $M$ and constants $\delta_{\epsilon}$ satisfying $\lambda_{\epsilon}\searrow v(T,x)$
for some $x\in M$ and $\delta_{\epsilon}\searrow0$, where $K$ is a fixed constant depending on almost Hermitian structure $(g,J,F)$ and let
$c>0$ be such that $t_0-cK>0$. For $\epsilon$ sufficiently small, therefore, $T_{c,\epsilon}\geq t_1F$ for some $t_1>0$, where $t_1$
can be chosen arbitrarily closed to $t_0$ if $c$ and $\epsilon$ are small enough \cite[Appendix C]{TaWaZhZh}.

The current $T_{c,\epsilon}$ is smooth off the $0$-dimensional set $E_c(T)$, that is, off a finite set of points. In a neighbourhood $U_{x_0}$
of any such point $x_0$, a locally symplectic form $\omega_{x_0}=d\tau_{x_0}$ on $U_{x_0}$, that is compatible with $J|_{U_{x_0}}$, where
$\tau_{x_0}\in \Omega_{\mathbb{R}}^1|_{U_{x_0}}$ \cite{Le0,Sik}. We may assume that $U_{x_0}$ is $\omega_{x_0}$-convex which is also called $J$-pseudo convex
\cite{ElGr,Pa,TaWaZhZh}. Moreover, we may assume that $U_{x_0}$ is a strictly $J$-pseudoconvex domain in the almost complex $4$-manifold $({\mathbb R}^4,J)$.
 By Lemma A.11 or Theorem A.31 in \cite{TaWaZhZh},
    there exists a strictly $J$-plurisubharmonic function $f$ such that
     $T_{c,\epsilon}=d\widetilde{\mathcal{W}}(f)=\widetilde{\mathcal{D}}^+_J(f)$ since $T_{c,\varepsilon}|_{U_{p_0}}$ is a $d$-closed positive $J$-$(1,1)$-current.
    Also we have the following estimate (Appendix A in \cite{TaWaZhZh}):
   \begin{equation}\label{estimate f}
  \|f\|_{L^2(U_{x_0},\varphi)}\leq\frac{1}{\sqrt{c}}\|\widetilde{\mathcal{W}}(f)\|_{L^2(U_{x_0},\varphi)},
   \end{equation}
   where $\varphi$ is a strictly $J$-plurisubharmonic function \cite{HaLa,Pa}
   satisfying $$\sum_{i,j}(\partial_{J_i}\bar{\partial}_{J_j}\varphi)\xi^i\bar{\xi}^j\geq c\sum_i|\xi^i|^2,$$ where $\xi\in\mathbb{C}^2$.

 Using a standard mollifying function as in \cite[page 147]{GiTr}, $f$ can be smoothed in neighborhoods of a family $f_t$ of strictly
 $J$-plurisubharmonic functions converging to $f$ and on an annular region surrounding $x_0$ the convergence of this sequence is uniformly in $C^k$
  for any $k$ (by Lemma 4.1 and the accompanying discussion in \cite{GiTr}). If $\rho$ is a standard cut off which is 1 on the exterior of the annulus and
  0 on the interior region. $\rho f+(1-\rho)f_t$ is a smooth $J$-plurisubharmonic function for $t$ sufficiently small which agrees with $f$ outside the annulus.
  Hence, the current $T_{c,\epsilon}$ is $\widetilde{\mathcal{D}}^+_J$-homologous to a smooth positive $J$-$(1,1)$-closed form $\tau_{c,\epsilon}$ for $\epsilon>0$ sufficient small \cite{TaWaZhZh}. Moreover, for any $t_1<t_0$,there is some $c$ and $\epsilon$ such that $\tau_{c,\epsilon}\geq t_1F$.
More details, see the proof of Theorem 1.1 in \cite{TaWaZhZh}.

This completes the proof of Theorem \ref{th3}. $\square$\\

As in the case of compact complex surfaces, using the method in proving Theorem \ref{th3},
 we have the following theorem which is similar to Theorem 12 in \cite{Bu}:
\begin{theo}\label{410}
Suppose that $(M,J)$ is a closed tamed almost complex $4$-manifold with $h_J^-=b^+-1$, where $\omega=F+d_J^-(v+\bar{v})$ is the $J$-taming
symplectic $2$-form on $M$, $v\in \Omega_J^{0,1}(M)$. Then $\widetilde{\omega}=F-d_J^+(v+\bar{v})$
is a $J$-almost K\"{a}hler form on $M$ modulo the image of ${\widetilde{\mathcal{D}}}^+_J$.
\end{theo}
\bp Let
 $$\widetilde{\omega}=\omega-d(v+\bar{v})=F-d_J^-(v+\bar{v}).$$

 In the proof of Theorem \ref{th3}, the cohomology class of $\tau_{c,\epsilon}$ is the same as that
$T$ which is in turn the same as that $\widetilde{\omega}-D$ for $D=\sum\gamma_i[D_i]$, where $[D_i]$ is an irreducible $J$-curve, and $\gamma_i\geq0$ \cite{LiZh1,LiZh2,Zh}. Here and subsequently the notation
is abused in the standard way by identifying a $d$-closed $J$-$(1,1)$-form with its image in $H_J^+(M)$ and vice verse; unless otherwise stated the
$J$-$(1,1)$-form representing a given class will always be that which has self-dual component a constant multiple of $F$ due to Proposition \ref{0110}, $h_J^-=b^+-1$ and Lemma \ref{67}. By Proposition \ref{01}, the notation $\psi\cdot\chi$ will be used to denote
  $$\int_M\psi\wedge\chi$$
   for $d$-closed
 $J$-$(1,1)$-forms $\psi$ and $\chi$. Thus, for each $\psi\in H_J^+(M)$,
 \begin{eqnarray}\label{40007}
\|\psi\|_{\widetilde{\omega}}^2:&=&2(\widetilde{\omega}\cdot\psi)^2(\widetilde{\omega}\cdot\widetilde{\omega})^{-1}-\psi\cdot\psi\nonumber\\
&=&(1+a)^{-1}(\widetilde{\omega}\cdot\psi)^2-\psi\cdot\psi
\end{eqnarray}
defines a norm on $H_J^+(M)$ since, by Proposition \ref{01}, the cup-product is negative definite on the orthogonal completement of $\widetilde{\omega}$,
that is,
$$H_J^+(M)\cong \mathbb{R}\widetilde{\omega}\oplus \mathcal{H}_g^-(M)\oplus {\widetilde{\mathcal{D}}}^+_J(C^{\infty}(M)_0).$$

If $t_1<t_0$ is such that $\tau_{c,\epsilon}\geq t_1F$, a simple calculation using the fact that $\tau_{c,\epsilon}=\widetilde{\omega}-D$ in $H_J^+(M)$
yields
 \begin{eqnarray}\label{40008}
0&\leq& \int_M(\tau_{c,\epsilon}-t_1F)^2\nonumber\\
&=&\int_M (\widetilde{\omega}-D-t_1F)^2\nonumber\\
&=&\int_M \{(\widetilde{\omega}-t_1F)^2+D\wedge D-2(\widetilde{\omega}-t_1F)\wedge D\}\nonumber\\
&=&[2(1+a)-4t_1(1+a)+2t_1^2]+D\cdot D-2(\widetilde{\omega}-t_1F)\cdot D.
\end{eqnarray}
Here, since $D=\Sigma\gamma_i[D_i]$, where $\gamma_i\geq0$, $[D_i]$ is an irreducible $J$-curve,
 \begin{eqnarray}\label{40009}
F\cdot D=\int_M F\wedge D=\sum_i\gamma_i\int_{D_i}F
\end{eqnarray}
is a non-negative term which is zero if $D=0$ in $H_J^+(M)$. Since $D$ is $d$-closed,
$\widetilde{\omega}\cdot D=F\cdot D$, so the last term on the right of the inequality is
$-2(1-t_1)F\cdot D$. Note that since $M$ is four-dimensional, hence by taking Poincar\'{e} duality,
we identify the $J$-curve classes with their Poincar\'{e} dual cohomology classes by abusing the notation.
Since
$$D\cdot D=-\|D\|_{\widetilde{\omega}}^2+(F\cdot D)^2(1+a)^{-1},$$
the inequality can be rewritten
\begin{eqnarray*}
\|D\|_{\widetilde{\omega}}^2&\leq& [2(1+a)-4t_1(1+a)+2t_1^2]\\
&+&\frac{F\cdot D}{1+a} [ F\cdot D-2(1+a)(1-t_1)].
\end{eqnarray*}
By Proposition \ref{0112}, the inequality
\begin{eqnarray*}
0\leq \widetilde{\omega}\cdot(\widetilde{\omega}-t_1F)=2(1+a)(1-t_1)-F\cdot D
\end{eqnarray*}
implies that
\begin{eqnarray*}
\|D\|_{\widetilde{\omega}}^2\leq [2(1+a)-4t_1(1+a)+2t_1^2]=2(t_0-t_1)[2(1+a)+t_0+t_1].
\end{eqnarray*}
Equality  holds only if $D$ is homologous to a multiple of $\widetilde{\omega}$.

Now choose a sequence of constants $t_i$, increasing to $t_0$ and corresponding costants $c_i$ and $\epsilon_i$
so that   $\tau_{c_i,\epsilon_i}\geq t_iF$, with $\tau_{c_i,\epsilon_i}=\widetilde{\omega}-D_{(i)}$ belonging to $H_J^+(M)$.
It follows from inequality above that the corresponding sequence of classes $D_{(i)}$ converging to zero in $H_J^+(M)$ and
therefore the corresponding representative $J$-$(1,1)$-forms (which have self-dual component of a constant multiple of $F$) are
converging to zero in $L^2(M)$; by elliptic regularity \cite{Au,GiTr} and choice of  representative, these forms are converging to zero in $C^0(M)$.
Since $\tau_{c_i,\epsilon_i}\geq t_iF$, it follows that for $i$ large enough  $J$-$(1,1)$-form $\tau_{c_i,\epsilon_i}+D_{(i)}$
is positive, i.e., $\widetilde{\omega}$ is $\widetilde{\mathcal{D}}^+_J$-homologous to a smooth positive closed $J$-$(1,1)$-form.

 This completes
the proof of Theorem \ref{410}.
\ep

Perhaps it is worth remarking that, as in compact complex surface \cite{Bu}, if $(M,J)$ is a tamed closed almost complex $4$-manifold, then there exists a real $2$-form $\omega$ on $M$ such that:\\
(1) $d\omega=0$;\\
(2) $P_J^+\omega$ is positive definite;\\
(3) $P_J^-\omega=d_J^-(u+\bar{u})$ for some $J$-$(0,1)$ form $u$;\\
(4) $0\leq h_J^-\leq b^+-1$  \cite{TaWaZh,TaWaZhZh}.\\

Suppose that $(M,g,J,F)$ is a closed almost Hermitian $4$-manifold. Let
$$\psi\in\Lambda^{1,1}_J\otimes L^2(M)$$
 be a weakly
${\mathcal{D}}^+_J$-closed $J$-$(1,1)$ current, and let $(\Lambda^{1,1}_J\otimes L^2(M))_{w}$ denote the space of
weakly ${\mathcal{D}}^+_J$-closed $J$-$(1,1)$ forms.
Through the remainder of this section, the fundamental form $F$ will be a fixed smooth positive weakly $\mathcal{D}_J^+$-closed $J$-$(1,1)$ form on $M$.
Similar to Lemma 2 in \cite{Bu} for compact complex surfaces. We have the following lemma:
\begin{lem}\label{402} Suppose that $(M,g,J,F)$ is a closed almost Hermitian $4$-manifold and $F$ is  weakly $\mathcal{D}_J^+$-closed.
For every $\psi\in (\wedge_J^{1,1}\otimes L^2(M))_w^0$, there exists $u_\psi\in \wedge_J^{0,1}\otimes L_1^2(M)$
such that
$$\widetilde{\psi}=\psi+d_J^+(u_\psi+\bar{u}_\psi)$$
is smooth and $*_g\widetilde{\psi}$ is $d$-closed, where $*_g$ is Hodge star operator with respect to the metric $g$.
\end{lem}
\bp
Let $\widetilde{\psi}$ be the $L^2$ projection of $\psi$ perpendicular to the image of $d_J^+$ with
respect to the  almost Hermitian metric $g(\cdot,\cdot)=F(\cdot,J\cdot)$. So, by Lemma \ref{311},
 $$\widetilde{\psi}=\psi+d_J^+(u_\psi+\bar{u}_\psi)$$
 for some $u_\psi\in\wedge_J^{0,1}\otimes L_1^2(M)$. Then
 $$\int_M\widetilde{\psi}\wedge*_g(\partial_Jv+\bar{\partial}_J\bar{v})=\int_M \widetilde{\psi}\wedge*_g(d_J^+(v+\bar{v})) =0$$
 for every $J$-$(0,1)$ form $v$ with coefficients in $L_1^2(M)$, where $*_g$ is Hodge star operator with respect to the metric $g$.
 Assume that
 $$\int_MF^2=2.$$
 This implies that $d(*_g\widetilde{\psi})=0$ weakly. Hence $*_g\widetilde{\psi}$ is also weakly ${\mathcal{D}}_J^+$-closed.
 By Proposition \ref{0110}, we have
 \begin{equation}\label{40004}
\widetilde{\psi}=C_{\psi}F+\beta_{\psi}+d_g^-(\gamma_{\psi})+\mathcal{D}_J^+(f_{\psi}),
 \end{equation}
 where $$C_{\psi}=\frac{1}{2}\int_MF\wedge \widetilde{\psi}$$ is a constant,  $$\beta_{\widetilde{\psi}}\in \mathcal{H}_g^-(M),
 \gamma_{\psi}\in\wedge_{\mathbb{R}}^1\otimes L_1^2(M),   f_{\psi}\in L_2^2(M)_0,$$ and
 $d_g^-(\gamma_{\psi})$ is weakly $\mathcal{D}_J^+$-closed. Since $*_g\widetilde{\psi}$ is also weakly ${\mathcal{D}}_J^+$-closed,
 it follows that
 $$0=\int_M*_g\widetilde{\psi}\wedge\mathcal{D}_J^+(f_{\psi})=\int_M*_g\mathcal{D}_J^+(f_{\psi})\wedge\mathcal{D}_J^+(f_{\psi}).$$
 This implies that $\mathcal{D}_J^+(f_{\psi})=0$. Thus, by \eqref{40004}, we have
  \begin{equation}\label{40005}
\widetilde{\psi}=C_{\psi}F+\beta_{\psi}+d_g^-(\gamma_{\psi}),
 \end{equation}
 and
   \begin{equation}\label{400051}
*_g\widetilde{\psi}=C_{\psi}F-\beta_{\psi}-d_g^-(\gamma_{\psi}).
 \end{equation}
 Obviously, $\wedge\widetilde{\psi}=2C_{\psi}$ is a constant, where $\wedge$ is the adjoint operator of
 $$L:f\mapsto fF$$
 for any $f\in L^2(M)$.
 Arguing as the proof of Lemma 2 in \cite{Bu}, we have
 $$d^*\widetilde{\psi}=0$$
  and
 $$dP_g^+\widetilde{\psi}=\frac{1}{2}(\wedge\widetilde{\psi})dF,$$ in the sense of currents,
 where $d^*=-*_gd*_g$. By elliptic regularity \cite{Au,GiTr}, it implies that $\widetilde{\psi}$ is smooth and weakly $\mathcal{D}_J^+$-closed.

 This completes the proof of Lemma \ref{402}.
\ep
It is similar to Theorem 38 of Harvey and Lawson \cite{HaLa0}, a weakened tamed symplectic structure can be characterized
 as follows (cf. \cite[Corollary 9]{Bu}):
\begin{theo}\label{18}
 Let $(M,g,J,F)$ be a closed almost Hermitian $4$-manifold. Suppose that  $F$ is weakly
 ${\mathcal{D}}_J^+$-closed. If $\varphi\in(\wedge_J^{1,1}\otimes L^2(M))_w^0$, then there exist $u_{\varphi}\in \wedge_J^{0,1}\otimes L_1^2(M)$, and $u_{\varphi}'\in\Omega_J^{0,1}$
  such that
 $$\varphi+d_J^+(u_\varphi+\bar{u}_\varphi)-2d_g^-(u_\varphi'+\bar{u}_\varphi')$$ is smooth and $d$-closed. In particular,
 $$F+d_J^+(u_F+\bar{u}_F)-2d_g^-(u_F'+\bar{u}_F')=C_FF-d_g^-(u_F'+\bar{u}_F')$$ is $d$-closed,
 which is called weakened tamed symplectic form on $M$,
  where $C_F$ is a constant, and $u_F$ and $u_F'\in \Omega_J^{0,1}(M).$
 \end{theo}
 \bp
Since $F$ is weakly $\mathcal{D}_J^+$-closed and $\varphi\in(\wedge_J^{1,1}\otimes L^2(M))_w^0$, by Lemma \ref{402}, there exists $u_\varphi\in\wedge_{J}^{0,1}\otimes L_1^2(M)$
 and $u_\varphi'\in\Omega_J^{0,1}$ such that
 \begin{eqnarray}\label{40006}
\widetilde{\varphi}:=\varphi+d_J^+(u_\varphi+\bar{u}_\varphi)=C_{\varphi}F+\beta_{\varphi}+d_g^-(u_\varphi'+\bar{u}_\varphi'),
 \end{eqnarray}
 and $*_g\widetilde{\varphi}$ is $d$-closed and smooth, where $C_\varphi$ is a constant, $\beta_{\varphi}\in \mathcal{H}_g^-(M)$.
Since $\varphi\in(\wedge_J^{1,1}\otimes L^2(M))_w^0$, it is easy to see that $\widetilde{\varphi}$ is also in $(\wedge_J^{1,1}\otimes L^2(M))_w^0$.
Hence,
 $$*_g\widetilde{\varphi}=C_\varphi F-d_g^-(u_\varphi'+\bar{u}_\varphi')-\beta_\varphi,$$
 it follows that
 $$C_\varphi F-d_g^-(u_\varphi'+\bar{u}_\varphi')$$
  is $d$-closed. Hence, by \eqref{40006}, we have
 \begin{eqnarray}\label{400014}
 \varphi+d_J^+(u_\varphi+\bar{u}_\varphi)-2d_g^-(u_\varphi'+\bar{u}_\varphi')
 \end{eqnarray} is smooth and $d$-closed.

If $\varphi=F$, it is in $(\wedge_J^{1,1}\otimes L^2(M))_w^0$. Then there are $u_F,u_F'\in \Omega_J^{0,1}$ and $C_F\neq0$ such that
 \begin{eqnarray}\label{400015}
\widetilde{F}:=F+d_J^+(u_F+\bar{u}_F)=C_FF+\beta_F+d_g^-(u_F'+\bar{u}_F').
\end{eqnarray}
Note that for any $\beta\in \mathcal{H}_g^-(M)$, there is
 $$\int_M\widetilde{F}\wedge\beta=\int_M(F+d_J^+(u_F+\bar{u}_F))\wedge\beta=\int_Md(u_F+\bar{u}_F)\wedge\beta=0.$$
It implies that $\beta_F=0$, that is,
 \begin{eqnarray*}
\widetilde{F}:=F+d_J^+(u_F+\bar{u}_F)=C_FF+d_g^-(u_F'+\bar{u}_F').
\end{eqnarray*}
By \eqref{400014} and \eqref{400015}, we get that
$$F+d_J^+(u_F+\bar{u}_F)-2d_g^-(u_F'+\bar{u}_F')=C_FF-d_g^-(u_F'+\bar{u}_F')=*_g\widetilde{F}$$
 is smooth and $d$-closed which is called weaken tamed symplectic form.

  This completes the proof of Theorem \ref{18}.
 \ep

\begin{rem}\label{19}
In Theorem \ref{18},  $C_F=0$ implies that $*_g\widetilde{F}=0$. In fact, if $C_F=0$,
then $$*_g\widetilde{F}=-d_g^-(u_F'+\bar{u}_F')$$ is smooth and $d$-closed. Thus, by Stokes Theorem, we have
$$\|d_g^-(u_F'+\bar{u}_F')\|^2=-\int_Md(u_F'+\bar{u}_F')\wedge d_g^-(u_F'+\bar{u}_F')=0.$$
Therefore, $C_F\neq0$.
\end{rem}

We now give a characterization of  almost K\"{a}hler $4$-manifolds.
\begin{theo}\label{17}  Let $(M,g,J,F)$ be a closed almost Hermitian $4$-manifold and $h_J^-=b^+-1$.
 Suppose that $F$ is weakly ${\mathcal{D}}_J^+$-closed and $C_F,u_F'$  defined in Theorem \ref{18}. If
 $$*_g\widetilde{F}=C_FF-d_g^-(u_F'+\bar{u}_F')$$
 satisfies that
 $$\int_M(*_g\widetilde{F})^2=\int_M(C_FF-d_g^-(u_F'+\bar{u}_F'))^2>0,$$
 then there exists a $J$-almost K\"{a}hler form on $(M,J)$.
 \end{theo}
\bp We may assume that, without loss generality,
$$\int_MF^2=2.$$
 Since $h_J^-=b^+-1$,
we have
$$H_g^+(M)\cap H_J^+(M)\oplus \mathcal{H}_g^-(M)\cong H_J^+(M),$$
$$ \dim(H_g^+(M)\cap H_J^+(M))=1.$$
By Theorem \ref{18} and Remark \ref{19},
$$\widetilde{F}=C_FF+d_g^-(u_F'+\bar{u}_F'),$$
where $C_F\neq0$ and $u_F'\in\Omega_J^{0,1}$. Also, $$*_g\widetilde{F}=C_FF-d_g^-(u_F'+\bar{u}_F')$$ is smooth and $d$-closed.
It is easy to see that $*_g\widetilde{F}\perp \mathcal{H}_g^-(M)$ with respect to the cup product and
$$\int_M*_g\widetilde{F}\wedge*_g\widetilde{F}=\int_M(C_FF-d_g^-(u_F'+\bar{u}_F'))^2>0.$$
Hence
$$*_g\widetilde{F}\in H_g^+(M)\cap H_J^+(M),$$
and
\begin{align*}
&\mathbb{R}\cdot*_g\widetilde{F}\oplus\mathcal{H}_g^-(M)+\mathcal{D}_J^+(C^{\infty}(M)_0)\\
\subseteq&\mathbb{R}\cdot*_g\widetilde{F}\oplus\mathcal{H}_g^-(M)+d(\Omega_{\mathbb{R}}^1(M))
=H_J^+(M).
\end{align*}
Arguing as in the proof of Theorem \ref{th3} (or the proof of Theorem 1.1 in \cite{TaWaZhZh}, or the proof of Theorem 11 in \cite{Bu}),
let
\begin{eqnarray*}
b:&=&\frac{1}{2}\int_M*_g\widetilde{F}\wedge*_g\widetilde{F}=\frac{1}{2}\int_M\widetilde{F}\wedge\widetilde{F}\\
&=&C_F^2-\frac{1}{2}\|d_g^-(u_F'+\bar{u}_F')\|_{L^2}^2\\
&=&C_F^2-C^2>0,\\
t_0:&=&C_F-\sqrt{C_F^2-(C_F^2-C^2)}=C_F-C>0,
\end{eqnarray*}
where $$C=\frac{\sqrt{2}}{2}\|d_g^-(u_F'+\bar{u}_F')\|_{L^2}\geq0.$$
It is easy to see that $C_F>0$.
Define
$$\varphi=*_g\widetilde{F}-t_0F,$$
which is weakly $\mathcal{D}_J^+$-closed.
$$\int_M\varphi^2=\int_M(*_g\widetilde{F})^2-2t_0\int_M*_g\widetilde{F}\wedge F+2t_0^2=0$$
and
\begin{align*}
\int_MF\wedge\varphi&=\int_MF\wedge\{(C_FF-d_g^-(u'_F+\bar{u}'_F))-t_0F\}\\
&=2C_F-2t_0=2C\geq0,
\end{align*}
with equality if and only if $*_g\widetilde{F}$ and $\widetilde{F}$ are $d$-closed, $*_g\widetilde{F}$ is $\mathcal{D}_J^+$-homologous to a constant multiple of $F$,
and $\widetilde{F}$ is a symplectic form on $M$.

Assume that $C_F-t_0=C>0.$ By Lemma \ref{0115} for each $n\in \mathbb{N}$, there is a smooth positive $J$-$(1,1)$-form $p_n$ and a smooth function $f_n$
such that
$$\|\varphi-\mathcal{D}_J^+(f_n)-p_n\|_{L^2(M)}<\frac{1}{n}.$$
Hence, the subsequence of positive $J$-$(1,1)$-forms $\{p_n\}$ in the sense of currents defines a positive $J$-$(1,1)$-current
$$p=\varphi+\mathcal{D}_J^+(f_0)=*_g\widetilde{F}-t_0F+\mathcal{D}_J^+(f_0),$$
where $f_0\in L_2^q(M)_0$, for some fixed $q\in (1,2)$ (c.f. the proof of Theorem \ref{th3}).
Thus, the $J$-almost K\"{a}hler current
$$P:=*_g\widetilde{F}+\mathcal{D}_J^+(f_0)=p+t_0F\geq t_0F,$$
in the sense of currents.

As done in the proof of Theorem \ref{th3} (or the proof of Theorem 1.1 in \cite{TaWaZhZh}),
using Lelong number of $P$ at $x\in M$ \cite{El,El1,LiZh2,TaWaZhZh}, Siu's decomposition theorem \cite{Si,TaWaZhZh} and Demailly approximating theorem \cite{De,TaWaZhZh},
for any
$$0<t_1<t_0,$$
we can construct a $J$-almost K\"{a}hler form $\tau\geq t_1F$.

 This completes the proof of Theorem \ref{17}.
\ep

 P. Gauduchon \cite{Ga} has shown that for a closed almost Hermtian $4$-manifold, there is a conformal rescaling of the metric, unique up to
a positive constant, such that the associated metric is Gauduchon metric, that is, there is a fundamental $2$-form such that
 $$\partial_J\bar{\partial}_JF=0.$$
 Since $\mathcal{D}_J^+$ can be viewed as a generalization of $\partial_J\bar{\partial}_J$, it is natural to pose the following question:
 \vspace{0.5cm} \\
{\bf Question 4.8.}
{\it Suppose that $(M,J)$ is a closed almost complex $4$-manifold. Does there always exist an almost Hermitian structure $(g,J,F)$ such
that the fundamental $2$-form $F$ is weakly $\mathcal{D}_J^+$-closed?}

\section{A Nakai-Moishezon criterion for almost complex $4$-manifolds}
 \setcounter{equation}{0}
 This section is devoted to studying the Nakai-Moishezon criterion by using $\widetilde{\mathcal{D}}^+_J$ (resp. ${\mathcal{D}}^+_J$) operator defined in
 \cite{TaWaZhZh}.
 The K\"{a}hler version of Nakai-Moishezon criterion established by Buchdahl \cite{Bu,Bu1} and Lamari \cite{La,La1} in dimenison 4, and by Demailly-Paun \cite{DePa} in arbitrary dimension, and almost K\"{a}hler version for rational 4-manifolds established by Li-Zhang \cite{LiZh,LiZh1,Zh}.\\

 We could define the $J$-curve cone $A_J(M)$ for an almost complex $4-$manifold $(M,J)$ \cite{LiZh1,LiZh2,Zh}:
 $$A_J(M)=\{\Sigma a_i[C_i]| a_i>0\},$$
where $C_i$ is an irreducible $J$-holomorphic subvariety on $M$ \cite{Ta}. Here an irreducible $J$-holomorphic subvariety
 is the image of a $J$-holomorphic map
 $$\phi:\Sigma\rightarrow M$$
  from a complex connected curve, where $\phi$ is an embedding off a finite set. More
 generally, an effective integral $J$-holomorphic subvariety is a finite of pairs $\{(C_i,m_i),1\leq i\leq n\}$, where each $C_i$ is an irreducible subvariety and each $m_i$
 is a non-negative integer. We sometimes say $J$-curve instead if $J$-holomorphic subvariety which is a $J$-invariant homology class in $H_+^J(M)$ \cite{LiZh}.

We now give a Nakai-Moishezon criterion for tamed almost complex $4$-manifolds (comparing with \cite[Theorem 14,Corollary 15, Theorem 16]{Bu}):
 \begin{theo}\label{th51}
 Suppose that $(M,g,J,F)$ is a closed almost Hermitian 4-manifold with $h^-_J=b^+-1$
 tamed by the symplectic form $\omega=F+d_J^-(v+\bar{v})$, where $v\in \Omega_J^{0,1}(M)$.
 If $\chi\in (\wedge_J^{1,1}\otimes L^2(M))_{\widetilde{w}}$ satisfies
 $$\int_M\chi^2>0, \int_MF\wedge\chi>0,$$
 and
 $$\int_D\chi>0$$
 for each irreducible $J$-curve (i.e., $J$-invariant homology class in $H_+^J(M)$) $D$ on $M$. Then $\chi$ is homologous to a smooth closed positive $J$-$(1,1)$ form
 modulo the image of $d_J^+$, and homologous to a smooth positive $J$-$(1,1)$ form modulo the image of $\widetilde{\mathcal{D}}^+_J$.
 \end{theo}
\bp
We may assume that, without loss of generality,
$$\int_MF^2=2.$$
Let
$$\widetilde{\omega}=F-d_J^+(v+\bar{v}).$$
By \eqref{0162},
$$\int_M\widetilde{\omega}^2=\int_M\omega^2=2(1+a).$$
By Theorem \ref{410}, $\widetilde{\omega}$ is a $J$-almost K\"{a}hler form since $h_J^-=b^+-1.$
It is easy to see that
$$H_J^+(M)\cong \mathbb{R}\cdot\widetilde{\omega}\oplus \mathcal{H}_g^-(M)+d(\Omega_{\mathbb{R}}^1(M)).$$
Let $K_J^c$ be the $J$-almost K\"{a}hler cone \cite{LiZh1,Zh}. Then $$K_J^c\subseteq H_J^+(M).$$

   Arguing as in the proof of Theorem 14 in \cite{Bu}, if $\chi$ is not homologous to a positive $(1,1)$-form modulo the image of $d_J^+$,
   by  Hahn-Banach Theorem \cite{Yo} there is a positive $d$-closed $J$-$(1,1)$-form $Q$ such that
    $$Q(\chi+d_J^+(u+\bar{u}))=Q(\chi)\leq0,$$
   for any $u\in\wedge_J^{0,1}\otimes L^2(M)$, $Q\in K_J^c$,
since $$h_J^-=b^+-1, \chi\in (\wedge_J^{1,1}\otimes L^2(M))_{\widetilde{w}}^0 \ {\rm (see\  Lemma \ \ref{67})}.$$

 As done in the proof of Theorem \ref{th3} (or Theorem 1.1 in \cite{TaWaZhZh}), for any $\epsilon>0$, there is a real  $J$-curve $\widetilde{D}=\sum\widetilde{v_i}\widetilde{[D_i]}$ \cite{BaPeVa,GrHa,McSa} such that $Q-\widetilde{D}$ is homologous modulo the image of $\widetilde{\mathcal{D}}^+_J$ to a smooth real closed
$J$-$(1,1)$ form $\tau$ with $\tau\geq-\epsilon F$ in the sense of currents. By Proposition  \ref{0112} and Corollary \ref{0113},
\begin{eqnarray}\label{501}
0&\leq& \int_M\chi\wedge(\tau+\epsilon F)=\int_M\chi\wedge(Q-\widetilde{D}+\epsilon F)\nonumber\\
&=&Q(\chi)-\sum\widetilde{v_i}\int_{\widetilde{D_i}}\chi+\epsilon\int_M\chi\wedge F,
\end{eqnarray}
\begin{eqnarray}\label{502}
\int_M\chi\wedge F=2C_{\chi}>0,
\end{eqnarray}
by the hypothesis,
$$\sum\widetilde{v_i}\int_{\widetilde{D_i}}\chi\geq0.$$
If $Q(\chi)<0$, a contradiction results by choosing $\epsilon$ sufficiently small so it may be assumed that $Q(\chi)=0$.

Now choose a sequence $\epsilon_i\searrow0$ and smooth real closed $J$-$(1,1)$ forms $\tau_i\geq-\epsilon_iF$ with $\tau_i$ homologous to $Q-\widetilde{D}_{(i)}$, modulo the image of $\widetilde{\mathcal{D}}^+_J$ for
some real effective $J$-curve $\widetilde{D}_{(i)}$. By Proposition \ref{0112},
\begin{eqnarray*}
0&\leq&\int_M(\tau_i+\epsilon_i F)\wedge\chi=\int_M(Q-\widetilde{D}_{(i)}+\epsilon_i F)\wedge \chi\\
&=&-\int_{\widetilde{D}_{(i)}}\chi+\epsilon_i\int_MF\wedge \chi,
\end{eqnarray*}
So,
\begin{eqnarray}\label{503}
\int_{\widetilde{D}_{(i)}}\chi\searrow0.
\end{eqnarray}
As done in the proof of Theorem \ref{410}, here and subsequently the notation is abused in the standard way by identifying
a $d$-closed $J$-$(1,1)$ form with its image in $H_J^+(M)$ and vice versa, unless otherwise stated, a $d$-closed $J$-$(1,1)$-form representing
a given class will always be that which has self-dual component a constant multiple of $F$ by Proposition \ref{0110}. In the same vein, the notation $\psi_1\cdot\psi_2$
will be used to denote
$$\int_M\psi_1\wedge\psi_2$$
for $d$-closed $J$-$(1,1)$ forms. Thus, for $\psi\in H_J^+(M)$, as in the proof of Theorem \ref{410},
$$\psi\mapsto 2(\widetilde{\omega}\cdot\psi)^2\cdot(\widetilde{\omega}\cdot\widetilde{\omega})^{-1}-\psi\cdot\psi
=\frac{1}{1+a}(\widetilde{\omega}\cdot\psi)^2-\psi\cdot\psi=\|\psi\|_{\widetilde{\omega}}^2$$
defines a norm on $H_J^+(M)$ since the cup-product pairing is negative definite on the orthogonal complement of $\widetilde{\omega}$ (Proposition \ref{01} with $h_J^-=b^+-1$).

If
\begin{eqnarray}\label{504}
\|\widetilde{D}_{(i)}\|_{\widetilde{\omega}}^2&=&\frac{1}{1+a}(F\cdot\widetilde{D}_{(i)})^2-\widetilde{D}_{(i)}\cdot\widetilde{D}_{(i)}\nonumber\\
&=&\frac{1}{1+a}(\widetilde{\omega}\cdot\widetilde{D}_{(i)})^2-\widetilde{D}_{(i)}\cdot\widetilde{D}_{(i)}
\end{eqnarray}
is not bounded independent of $i$, a subsequence can be found with
 $$\|\widetilde{D}_{(i_j)}\|_{\widetilde{\omega}}\rightarrow\infty.$$
 The cohomology class
$ \frac{\widetilde{D}_{(i_j)}}{\|\widetilde{D}_{(i_j)}\|_{\widetilde{\omega}}}$ can be assumed to converge to some closed $J$-$(1,1)$ form $\widetilde{D}\in H_J^+(M)$ of norm $1$, with $\chi\cdot \widetilde{D}=0$ due to \eqref{503}.
Hence $\widetilde{D}\cdot\widetilde{D}<0$ by  Proposition \ref{0112} since $\widetilde{D}$ is non-zero in $H_J^+(M)$.  In fact,
If $\widetilde{D}\cdot\widetilde{D}\geq0$ then $F\cdot\widetilde{D}\geq0$ since $F\cdot\widetilde{D}_{(i_j)}>0$ and 
$$0=\chi\cdot\widetilde{D}=(\chi\cdot\chi)^{\frac{1}{2}}\cdot(\widetilde{D}\cdot\widetilde{D})^{\frac{1}{2}}\geq0.$$
It follows that $\chi$ and $\widetilde{D}$ are linearly dependent modulo the image of $\widetilde{\mathcal{D}}_J^+$, giving a contradiction.
Note that $Q-\widetilde{D}_{(i_j)}$ is homologous modulo the image of $\widetilde{\mathcal{D}}_J^+$ to $J$-$(1,1)$ form $\tau_{i_j}\geq-\epsilon_{i_j}F$,
hence
$$\int_M(Q-\widetilde{D}_{(i_j)}+\epsilon_{i_j}F)\wedge(Q-\widetilde{D}_{(i_j)}+\epsilon_{i_j}F)>0,$$
it follows that $\widetilde{D}\cdot\widetilde{D}\geq0$, a contraction. Therefore, $\{\|\widetilde{D}_{(i)}\|_{\widetilde{\omega}}\}$ is bounded and by passing to a
subsequence it can be assumed that $\widetilde{D}_{(i)}$ converges to some $\widetilde{D}$ with
$$\widetilde{D}\cdot\widetilde{D}\leq0, F\cdot\widetilde{D}\geq0,$$
and
$$\chi\cdot\widetilde{D}=0.$$
By Proposition \ref{0112} again, the inequality
$$(Q-\widetilde{D})\cdot(Q-\widetilde{D})\geq0$$
 and
 $$\chi\cdot(Q-\widetilde{D})=0$$ imply that
$Q=\widetilde{D}$ in $H_J^+(M)$, due to $Q(\chi)=0$. In fact, the real effective $J$-curve $\widetilde{D}_{(i)}$ is of the form $\widetilde{D}_{(i)}=\sum v_{ik}\widetilde{D}_{(i)k}$,
 where $\widetilde{D}_{(i)k}$ is the $k$-th $J$-irreducible $1$-dimensional component of the $C_i$-upperlevel set of $v(Q,-)$, with $v_{ik}$ the Lelong
 number of $Q$ at the generic point of $\widetilde{D}_{(i)k}$. As $\epsilon_i$ tends to $0$, corresponding constants $C_i$ can be assumed to converge monotonically to 0, $\widetilde{D}_{(i)}\subset\widetilde{D}_{(k)}$ for $i<k$ and the constants $v_{ik}$ appearing in $\widetilde{D}_{(i)}$ also appear as the coefficients in $\widetilde{D}_{(k)}$. Since $\chi\cdot E>0$, by \eqref{0135}, for every $J$-curve $E$, it follows $\chi\cdot\widetilde{D}_{(i)}\leq\chi\cdot\widetilde{D}$ for $i<k$ with equality if and only if  $\widetilde{D}_{(i)}=\widetilde{D}_{(k)}$.

Since $\chi\cdot\widetilde{D}=0$, it therefore follows that $\widetilde{D}_{(i)}$ is homologous to 0 for every $i$, as is the limit $\widetilde{D}$. Therefore,
$Q=0$ in $H_J^+(M)$ implying $F\cdot Q=0$, a contradiction since $Q$ is a positive $d$-closed $J$-$(1,1)$ form.
It therefore follows that $\chi$ is homologous to a closed
positive $J$-$(1,1)$ form on $M$ modulo the image of $d_J^+$, so from Theorem \ref{410} (or Theorem 12 in \cite{Bu}) follows that the form is actually homologous to a closed positive $(1,1)$ form on $M$.
In particular, if $\chi\in H_J^+(M)$ can be represented by a smooth positive $J$-$(1,1)$ form if and only if $\chi\cdot\chi>0$, $F\cdot\chi>0$
and $\chi\cdot \widetilde{D}>0$ for every $J$-curve $\widetilde{D}\subset M$.\\

Now, we prove that $\chi$ is $\widetilde{\mathcal{D}}_J^+$-homologous to the positive $J$-$(1,1)$ form as done in the proof of Theorem 16 in \cite{Bu}. By the result above, there exists
$u_{\chi}\in \wedge_J^{0,1}\otimes L_1^2(M)$ such that
$$\widetilde{\chi}=\chi+d_J^+(u_{\chi}+\bar{u}_{\chi})$$
is a $J$-almost K\"ahler form on $M$. It follows that
$$\int_M(\chi+d_J^+(u_{\chi}+\bar{u}_{\chi}))\wedge d_J^+(u'+\bar{u}')=0,$$
for any $u'\in \wedge_J^{0,1}\otimes L_1^2(M)$. Hence,
$$\int_M\chi\wedge d_J^+(u_{\chi}+\bar{u}_{\chi})=-\int_M (d_J^+(u_{\chi}+\bar{u}_{\chi}))^2.$$
Thus,
$$\int_M\widetilde{\chi}^2=\int_M\chi^2-\int_M(d_J^+(u_\chi+\bar{u}_\chi))^2.$$
As done in the proof of Lemma \ref{65}, by \eqref{40001}-\eqref{40003} and Stokes Theorem,
we have
\begin{eqnarray}\label{40010}
\int_M\widetilde{\chi}^2=\int_M\chi^2+\int_M(d_J^-(u_\chi+\bar{u}_\chi))^2\geq \int_M\chi^2,
\end{eqnarray}
since
$$d_J^+(u_\chi+\bar{u}_\chi)=d(u_\chi+\bar{u}_\chi)-d_J^-(u_\chi+\bar{u}_\chi).$$
If equality holds, then $d_J^-(u_\chi+\bar{u}_\chi)=0$. By Lemma \ref{65}, there exists a $f_\chi\in L_2^2(M)_0$ such that $$d_J^+(u_\chi+\bar{u}_\chi)=d(u_\chi+\bar{u}_\chi)=\mathcal{\widetilde{D}}_J^+(f_\chi).$$
It is easy to see that
\begin{eqnarray}\label{40011}
\int_M\widetilde{\chi}\wedge F>0
\end{eqnarray}
and
\begin{eqnarray}\label{40012}
\int_D\widetilde{\chi}=\int_D\chi>0
\end{eqnarray}
for any irreducible $J$-curve $D$.

Without loss of generality, we may assume that
\begin{eqnarray*}
d_J^-(u_{\chi}+\bar{u}_{\chi})\neq0,
\end{eqnarray*}
that is,
\begin{eqnarray*}
d\chi\neq0.
\end{eqnarray*}
Define
  \begin{eqnarray}\label{508}
 t_0=1-\sqrt{1-\frac{\int_M\chi^2}{\int_M\widetilde{\chi}^2}}.
  \end{eqnarray}
 \eqref{40010} and \eqref{508} imply that $\psi=\chi-t_0\widetilde{\chi}$
 satisfies
  \begin{eqnarray*}
\int_M\psi^2&=&\int_M\chi^2-2t_0\int_M\widetilde{\chi}\chi+t_0^2\int_M\widetilde{\chi}^2\\
&=&\int_M\widetilde{\chi}^2-\int_M(d_J^-(u_{\chi}+\bar{u}_{\chi}))^2-2t_0\int_M\widetilde{\chi}^2\\
&+&2t_0\int_M\widetilde{\chi}\wedge d_J^+(u_{\chi}+\bar{u}_{\chi})+t_0^2\int_M\widetilde{\chi}^2\\
&=&(1-t_0)^2\int_M\widetilde{\chi}^2-\int_M(d_J^-(u_{\chi}+\bar{u}_{\chi}))^2\\
&=&\int_M\widetilde{\chi}^2-\int_M{\chi}^2-\int_M(d_J^-(u_{\chi}+\bar{u}_{\chi}))^2\\\
&=&0,
\end{eqnarray*}
and
 \begin{eqnarray*}
 \int_M\widetilde{\chi}\wedge\psi&=&(1-t_0)\int_M\widetilde{\chi}^2-\int_M\widetilde{\chi}\wedge d_J^+(u_{\chi}+\bar{u}_{\chi})\\
 &=&(1-t_0)\int_M\widetilde{\chi}^2>0.
\end{eqnarray*}
Since $\widetilde{\chi}$ is an almost K\"{a}hler form, by the same arguments as those leading to Theorem \ref{th3} (or Theorem 1.1 in \cite{TaWaZhZh}),
 there is a sequence of smooth function $f_n$ and positive $(1,1)$-form $p_n$ such that
 $$\|\psi+\widetilde{\mathcal{D}}_J^+(f_n)-p_n\|\rightarrow0,$$
 with $f_n$ converging weakly in $L_1^{\frac{4}{3}}\cap L^2$ and strongly in $L^q$ for some $q\in (1,2)$ to a function $f$ and with $p_n$ converging
 in the sense of currents to some positive $(1,1)$-current $p\in\wedge_J^{1,1}\otimes L^1(M)$.
 The current
 $$P:=\widetilde{\mathcal{D}}_J^+(f)=p-\psi$$
 is then closed and almost positive with $P\geq-\psi$.
 Arguing exactly as in the proof of Theorem \ref{th3} (or the proof of Theorem 1.1 in \cite{TaWaZhZh}), for some constant $K$ depending on the curvature
 of $(M,g,J,F)$. Given $c>0$ with $t_0-cK>0$, there is an $\mathbb{R}$-linear combination $D_c$ of $J$-curves on $M$ with positive coefficients and
 a family of smooth functions $f_{c,\epsilon}$ with
 \begin{eqnarray}\label{515}
\widetilde{\mathcal{D}}_J^+(f_{c,\epsilon})-D_c\geq-\psi-(cK+\delta_\epsilon)\widetilde{\chi}
\end{eqnarray}
such that $\delta_\epsilon\searrow0$ as $\epsilon\searrow0$; (as before $D_c$ is here identified with a represented smooth closed $(1,1)$-form).
Since
 \begin{eqnarray*}
\int_M&&(\psi+\widetilde{\mathcal{D}}_J^+(f_{c,\epsilon})-D_c)^2\\
&=&D_c\cdot D_c-2(1-t_0)D_c\cdot\chi\\
&=&-\|D_c\|^2-2\frac{\chi\cdot D_c}{\widetilde{\chi}\cdot\widetilde{\chi}}\cdot\widetilde{\chi}\cdot(\psi-D_c+(cK+\delta_\epsilon)\widetilde{\chi})
+2(cK+\delta_\epsilon)
\end{eqnarray*}
is converging to a non-negative number as $c,\epsilon\rightarrow0$, it follows exactly as in the proof of Theorem \ref{410}.
(representative $J$-$(1,1)$ forms for) $D_c$ must be converging to 0 in $C^0(M)$. Consequently the inequality
$$\chi+\widetilde{\mathcal{D}}_J^+(f_{c,\epsilon})\geq (t_0-cK-\delta_{\epsilon})\widetilde{\chi}+D_c$$
implies that $\chi+\widetilde{\mathcal{D}}_J^+(f_{c,\epsilon})$ is positive for $c,\epsilon$ sufficiently small.

This completes the proof of Theorem \ref{th51}.
\ep

  \begin{rem}\label{0140}
Theorem \ref{th51} is still true if $\chi$ is a $d$-closed $J$-$(1,1)$ current satisfying
$$\chi\cdot\chi>0, F\cdot\chi>0$$
 and
$$\int_D\chi>0$$
 for each $J$-curve $D\subset M$ \cite[Corollary 15]{Bu}.

The hypothesis $F\cdot\chi>0$ can not be omitted in general even if $\widetilde{M}$ has at least one curve. For examples, \cite[Page 299]{Bu},
if $M$ is a $K3$ surface with no divisor at all (c.f. \cite{Bu,Bu3,Bu4}, \cite[Page 594]{GrHa}, \cite{Zhe}).
$$\pi:\widetilde{M}\rightarrow M$$
 is the blow up at $p\in M$ and $D=\pi^{-1}(p)$ is the exceptional divisor. If $\omega_0$ is the K\"{a}hler form and $\rho$ is a smooth closed $(1,1)$-form on $\widetilde{M}$ with $[\rho]=[D]$ in
$H_J^+(\widetilde{M})$, $$\chi:=-\pi^*\omega_0-\epsilon\rho$$ satisfies $\chi\cdot\chi>0$ for $\epsilon>0$ sufficient small, $\chi\cdot D>0$, due to none of curves on $M$.
But $\chi$ can never be homologous to a positive $J$-$(1,1)$-form modulo the image of $d_J^+$ since $$\int_{\widetilde{M}}\pi^*\omega_0\wedge\chi<0.$$
 \end{rem}
In Theorem \ref{th51} (or Corollary 15, Theorem 16 in \cite{Bu}), the classical Nakai-Moishezon criterion for a closed tamed $4$-manifold $(M,J)$ with
$h_J^-=b^+-1$, was generalized to yield a characterization  of the set of classes in $H_J^+$ which can be represented by an almost K\"{a}hler form.
Recall that if $J$ is integrable, then $b^1$ is even. Similar to Main Theorem in \cite{Bu1}, we have:
 \begin{theo}\label{th53}
 Let $(M,J)$ be a closed almost complex 4-manifold equipped with  weakly ${\mathcal{D}}^+_J$-closed the fundamental $2$-form $F$ and $h_J^-=b^+$.
 Let $\varphi$ be a smooth real weakly ${\mathcal{D}}^+_J$-closed $J$-$(1,1)$ form satisfying
 $$\int_M\varphi\wedge\varphi>0,\int_M\varphi\wedge F>0, \int_D\varphi>0$$
 for every irreducible $J$-curve $D\subset M$ with $D\cdot D<0$.
If there is a non-zero effective integral $J$-curve $E$ on $M$ with $E\cdot E=0$,
then there is a smooth function $f$ on $M$ such that $\varphi+{\mathcal{D}}^+_J(f)$ is positive.
 \end{theo}
 \begin{rem}
 (1) Theorem \ref{th51} differs from this only in that it assumes  $h_J^-=b^+-1$ and that $$\int_D\varphi>0$$
 for every $J$-curve $D\subset M$; however, this inequality must hold for each irreducible $J$-curve $D\subset M$
 with $D\cdot D\geq0$ by Proposition \ref{0110}. Note that the $J$-$(1,1)$ component of taming symplectic weakly ${\mathcal{D}}^+_J$-closed.\\
 (2) When $J$ is integrable, $(M,J)$ is a complex surface. For compact surface $(M,J)$, if $M$ has algebraic dimension $2$, it is well-known
 that $M$ is projective \cite{BaPeVa}, hence $(M,J)$ is K\"{a}hler surface, $h_J^-=b^+-1$; if $(M,J)$ has algebraic dimension $1$, it follows that
 $(M,J)$ is an elliptic surface \cite{BaPeVa} and therefore such a divisor $E$ with $E\cdot E=0$ exists; if $(M,J)$ has algebraic dimension $0$, then
 there are finitely many irreducible curves on $M$ \cite{BaPeVa} so that for $\epsilon$ sufficiently small, this implies that there is an effective
 non-zero integral divisor $E$ on $M$ with $E\cdot E=0$ (for more details; see \cite[page 1536-1537]{Bu1}).\\
 (3) It is natural to ask the following question:
 for any closed almost complex $4$-manifold $(M,J)$, could one find a non-zero effective integral $J$-curve $E$ with $E\cdot E=0$.
 \end{rem}
{\bf Proof of Theorem \ref{th53}.} Our approach is along the lines to give the proof of Main Theorem in \cite{Bu1}.
Suppose that $(M,g,J,F)$ a closed almost Hermitian $4$-manifold with $h_J^-=b^+$, where $F$ is weakly ${\mathcal{D}}^+_J$-closed.
Since
$$\Omega_J^-(M)=\mathcal{H}_J^-(M)+d_J^-(\Omega_{\mathbb{R}}^1(M)),$$
$$\Omega_g^+(M)=\mathcal{H}_g^+(M)+d_g^+(\Omega_{\mathbb{R}}^1(M)),$$
and $h_J^-=b^+$, then
\begin{align}\label{510}
\Omega_J^+(M)=\mathbb{R} F+P_g^+(d_J^+(\Omega_{\mathbb{R}}^1(M))\oplus\mathcal{H}_g^-(M)+d_g^-(\Omega_{\mathbb{R}}^1(M)).
\end{align}
It follows that
\begin{align}\label{511}
H_J^+(M)=\mathcal{H}_g^-(M)+d(\Omega_{\mathbb{R}}^1(M)),
\end{align}
Hence, the almost K\"{a}hler cone, $K_J^c$, is empty.
For any $$\psi\in (\wedge_J^{(1,1)}\otimes L^2(M))_{w}= (\wedge_J^{(1,1)}\otimes L^2(M))_{w}^0,$$
due to $h_J^-=b^+$, by \eqref{510}, \eqref{511} and Proposition \ref{0110}, $\psi$ could be written as
$$\psi=C_{\psi}F+\beta_{\psi}+{\mathcal{D}}^+_J(f_{\psi})+d_g^-(\gamma_{\psi}),$$
where $C_{\psi}$ is a constant,
 $$\beta_{\psi}\in\mathcal{H}_g^-(M), f_{\psi}\in L_2^2(M)_0, \gamma_{\psi}\in\wedge_{\mathbb{R}}^1\otimes L_1^2(M).$$
Hence, $F\wedge(\psi-{\mathcal{D}}^+_J(f_{\psi}))$ is a constant multiple of $F^2$. By Proposition \ref{01}, the intersection form
on $H_J^+(M)$ is negative definite.
By Theorem \ref{18} and Remark \ref{19}, we have $d$-closed $J$-$(1,1)$ form
$$*_g\widetilde{F}=C_FF-d_g^-(u_F'+\bar{u}_F')$$
where $C_F\neq0$, $u_F'\in\Omega_J^{0,1}(M)$. Let
$$\sigma_0=\frac{1}{C_F}*_g\widetilde{F}=F-\frac{1}{C_F}d_g^-(u_F'+\bar{u}_F'),$$
then $F\wedge \sigma_0=F\wedge F$ and $[\sigma_0]=0$ in $H_J^+(M)$ due to \eqref{511}.
The harmonic representation of a closed $J$-$(1,1)$ form $\sigma$ on $M$ satisfying
$$F\wedge \sigma=cF\wedge F$$
for some constant $c$ is then $\sigma-c\sigma_0$. This form is anti-self-dual since $$(\sigma-c\sigma_0)\wedge F=0.$$

Let $Z\subset (M,J)$ be an irreducible $J$-curve, then $Z$ is a $J$-analytic subset defined  by Elkhadhra \cite{El}.
If $Z$ is defined by $f=0$, where $f$ is of class $C^3$, $\overline{\partial}_Jf=0$, $Z$ and ${\partial}_Jf\neq0$, we have
a generalized Poincar\'{e}-Lelong formula \cite{El1} as follows:
\begin{align}\label{512}
\frac{1}{4\pi}d_J^+Jd\log f^2=\frac{\sqrt{-1}}{2\pi}\partial_J\overline{\partial}_J\log f^2=[Z]+R_J(f),
\end{align}
where $[Z]$ is the current of integration on the $J$-curve $Z$, $R_J(f)$ is a $J$-$(1,1)$ current which has $L_{loc}^{\alpha}$ integrable
whose $\alpha<2$. Moreover, $R_J(f)$ is weakly ${\mathcal{D}}^+_J$-closed.  $R_J(f)=0$ when the structure $J$ is integrable, that is called
Poincar\'{e}-Lelong theorem \cite{GrHa}.

Recall that any almost complex $4$-manifold has local symplectic property \cite{Le}. Hence for any $p\in (M,J)$, we choose a neighborhood $U_p$
of $p$. We may assume without loss generality that $U_p$ is a star sharped strictly $J$-pseudoconvex open set. Let $T$ be a strictly closed positive
$J$-$(1,1)$ current on $M$. By Poincar\'{e} lemma, $T|_{U_p}=dA$ on $U_p$. Note that $T$ is $J$-$(1,1)$ type, so $d_J^-A=0$. Applying Lemma \ref{65},
there is a smooth function $f_T$ such that $T|_{U_p}=d\mathcal{W}(f_T)=(d\widetilde{\mathcal{W}}(f_T))$. Since $(M,J)$ is locally almost K\"{a}hler $4$-manifold,
$\mathcal{W}(f_T)=\widetilde{\mathcal{W}}(f_T)$ on $U_p$. Hence, $T|_{U_p}=d\mathcal{W}(f_T)$ by Theorem A.31 for $(\widetilde{\mathcal{W}},d_J^-)$-problem of \cite{TaWaZhZh}. When $U_p$ is very small, on $U_p$  there exists Darboux coordinate chart
$z_1,z_2$ \cite{McSa0} with standard complex structure $J_0$ satisfying $J_0=J(p)$ on $T_pM\cong \mathbb{R}^4$. Since $d\mathcal{W}(f_T)={\mathcal{D}}^+_J(f_T)$ is smooth and strictly
positive $(1,1)$-form, ${\mathcal{D}}^+_J(f_T)$ can be regarded as a local symplectic form on $U_p$. Hence the complex coordinate is also Darboux
coordinate on $U_p$ for ${\mathcal{D}}^+_J(f_T)$, that is,
${\mathcal{D}}^+_J(f_T)$ are $J$ and $J_0=J(p)$ compatible \cite[Theorem A.11]{TaWaZhZh}. Thus,
$$f_T|_{U_p}=|z_1|^2+|z_2|^2,$$
\begin{align}\label{513}
{\mathcal{D}}^+_J(f_T|_{U_p})=d\mathcal{W}(f_T|_{U_p})=2\sqrt{-1}\partial_{J_0}\overline{\partial}_{J_0}f_T|_{U_p}.
\end{align}

By an effective integral $J$-curve $D\subset (M,J)$, we can define a complex line bundle $L_D$ over $(M,J)$,
the corresponding first Chern class $[c_1(L_D)]$  is a constant multiple of $[\sigma_D]$, where $\sigma_D$ is the
Poincar\'{e} dual class \cite{Ta0}.

A real $J$-curve is by definition a finite formal sum of the form
$$D=\sum v_iD_i,$$
 where $D_i\subset (M,J)$ is an irreducible $J$-curve on $(M,J)$,
$v_i$ is a real number, $D$ is effective if $v_i\geq0$ for all $i$ \cite{LiZh1,LiZh2,Zh}. The intersection form on $H_J^+(M)$ is denoted product symbol.
The $E\cdot E$ is the self intersection number of an effective integral $J$-curve $E$ in $(M,J)$. The notation extends to $\mathbb{R}$-linearly to all $J$-curve,
and is further extended to denote the pairing between weakly ${\mathcal{D}}^+_J$-closed $J$-$(1,1)$-forms:
$$\varphi\cdot\psi:=\int_M\varphi\wedge\psi,$$
for $\varphi,\psi\in (\wedge_{J}^{1,1}\otimes L^2(M))_w$. For a real $J$-curve $E$, the notation $\sigma_E\in H_J^+(M)$ will be used to denote
Poincar\'{e} dual of $E\in H_+^J(M)$. $\varphi\cdot E$ may also be used in place of
$$\int_E\varphi=\int_M\varphi\wedge\sigma_E.$$
We have the following
lemma (c.f. Lemma 1 in \cite{Bu1}):
\begin{lem}\label{55}
Let $E\subset (M,J)$ be an effective integral $J$-curve such that $E\cdot E=0$. Then for any $\epsilon>0$, there is a smooth function $g$
such that $$\sigma_E+{\mathcal{D}}^+_J(g)\geq -\epsilon F.$$
\end{lem}
\bp
If  there is no smooth function such that
$$\sigma_E+{\mathcal{D}}^+_J(g)\geq -\epsilon F$$
is positive in a neighbourhood of $E$, the Hahn-Banach Theorem \cite{Yo} (for Hilbert space $(\wedge_J^{1,1}(M)\otimes L^2(M))_w^0$) implies that
the existence of current $T$  and a constant $C$ such that
$$T(\sigma_E+\mathcal{D}^+_J(g)+\epsilon F)\leq C,$$
for every smooth function $g$ and $T(\psi)> C$ for every smooth $2$-form  $\psi$ whose $J$-$(1,1)$-component is positive in a neighhorhood of $E$.

It follows immediately that $T$ is $J$-$(1,1)$ current and $C$ must be non-positive such that
$$T(\sigma_E+{\mathcal D}_J^+(g)+\epsilon F)=T(\sigma_E+\epsilon F)\leq C.$$
 $T$ is weakly
$\mathcal{D}^+_J$-closed in the sense of currents and $$T(\psi)>0$$ in a neighborhood of $E$ for any smooth $J$-$(1,1)$ form $\psi$. The support of $T$ must
be contained in $E$. By Proposition \ref{0110}, as complex surface (see Lemma 32 of \cite{HaLa}), it follows that $T=\Sigma h_iE_i$, where $h_i$ is a non-negative constant and $E_1,E_2,\cdots$ the irreducible component of $E$. Since $E\cdot E=0$ and $h_J^-=b^+$, $[E]=[\sigma_E]=0$ in $H_J^+(M)$. Hence
$$E_i\cdot E=\int_{E_i}\sigma_E=0,$$
 for all $i$, and this gives a contradiction since
$$C\geq T(\sigma_E+\epsilon F)=\sum h_i\int_{E_i}\sigma_E+\sum h_i\int_{E_i}\epsilon F>0.$$

Since $E$ is an effective integral curve, $\sigma_E>-\epsilon F$, $\sigma_E=d\mathcal{W}(f_E)$, locally, where $f_E\in L_2^q(M)$, $q\in(1,2)$ and
strictly $J$-plurisubharmonic. By \eqref{513}, locally,
$$\mathcal{D}^+_J(f_E)=d\mathcal{W}(f_E)=2\sqrt{-1}\partial_{J_0}\overline{\partial}_{J_0}f_E.$$
Note that for any $p$ in an open neighbourhood $U$ of $E$, by rescaling, it can be assumed that the local symplectic form
 $$\omega_p\geq \frac{1}{2}F,\sigma_E>-\epsilon F$$ in an open neighbourhood $U$ of $E$. We can assume that
$$\{x\in M| |f_E|\leq1\}\subset U.$$
Let $h$ be a smooth convex increasing function on $\mathbb{R}$ such that $0\leq h'(t)\leq1$ for all $t$, with $h(t)=t$ for $t\geq0$ and with $h(t)=-1$
for $t\leq-1$. Then
\begin{align*}
2\sqrt{-1}\partial_{J_0}\overline{\partial}_{J_0}(h\log|f_E|^2)&=h'(\log|f_E|^2)2\sqrt{-1}\partial_{J_0}\overline{\partial}_{J_0}f_E\\
&+h'(\log|f_E|^2)\cdot2\sqrt{-1}\partial_{J_0}(\log|f_E|^2)\wedge\partial_{J_0}(\log|f_E|^2)\\
&\leq \chi'(\log|f_E|^2)2\sqrt{-1}\partial_{J_0}\overline{\partial}_{J_0}f_E.
\end{align*}
So
$$2\sqrt{-1}\partial_{J_0}\overline{\partial}_{J_0}f_E\geq (1-\chi'(\log|f_E|^2))\cdot2\sqrt{-1}\partial_{J_0}\overline{\partial}_{J_0}f_E\geq -\epsilon F,$$
as required.

 This completes the proof of Lemma \ref{55}.
\ep

We need the following Lemma (cf. Lemma 2 in \cite{Bu1}):
\begin{lem}\label{56}
Suppose that
$$\psi\in(\wedge_J^{1,1}\otimes L^2(M))_w^0, \psi\cdot\psi=0, F\cdot\psi\geq0 \ and\  \psi\cdot D\geq0$$
for every effective $J$-curve
$D\subset (M,J)$. Then for any $\epsilon>0$, there is a smooth function $f_{\epsilon}$
such that $$\psi+{\mathcal{D}}^+_J(f_{\epsilon})\geq -\epsilon F.$$
\end{lem}
\bp
By Lemma \ref{0115} (for complex surface case, cf. \cite[Lemma 7]{Bu}), $\psi$ can be approximated arbitrarily closely in $L^2$ norm by forms of
the type $p-{\mathcal{D}}^+_J(f)$, where $p$ is smooth and positive and $f$ is smooth. Following exactly the same argument as used in the proof of
Theorem \ref{th3}, a sequence of smooth functions $f_n$ and smooth positive $J$-$(1,1)$ forms $p_n$ can be found such that
$$\|\psi+{\mathcal{D}}^+_J(f_n)-p_n\|_{L^2(g)}$$
is converging to 0 and $f_n$ is converging in $L^1$ to define an almost positive closed $J$-$(1,1)$ current
$$P={\mathcal{D}}^+_J(f_{\infty})\geq-\psi.$$
Applying the same arguments as in the proofs of Theorem \ref{th3} and the second part of Theorem \ref{th51} shows that for any $\epsilon>0$, there is
a real effective $J$-curve $D_{\epsilon}$ and a smooth function $f_{\epsilon}$ such that
$$\sigma_{D_{\epsilon}}+{\mathcal{D}}^+_J(f_{\epsilon})\geq-\psi-\epsilon F.$$
The construction of $D_{\epsilon}$ is such that $D_{\epsilon'}\geq D_{\epsilon}$ for ${\epsilon'}<{\epsilon}$ and the coefficients of an irreducible
component common to both $D_{\epsilon'}$ and  $D_{\epsilon}$ are the same in both.

Now take a sequence of positive numbers $\epsilon$ converging monotonically to 0. Since
$$\chi_{\epsilon}=\epsilon F+\psi-\sigma_{D_{\epsilon}}+{\mathcal{D}}^+_J(f_{\epsilon})$$
is positive,
$$0\leq \chi_{\epsilon}\cdot \chi_{\epsilon}=\epsilon^2F\cdot F+D_{\epsilon}\cdot D_{\epsilon}+2\epsilon F\cdot\psi-F\cdot D_{\epsilon}-\psi\cdot D_{\epsilon}.$$
The hypothesis on $\psi$ and negativity of the intersection form restricted to $H_J^+(M)$ imply that the cohomology class $[D_{\epsilon}]\in H_J^+(M)$
are uniformly bounded. After passing to a subsequence, if necessary, the corresponding sequence of harmonic representative can be assumed  to converge smoothly.
Moreover, the inequality
$$0\leq F\cdot \chi_{\epsilon}=\epsilon F\cdot F+F\cdot\psi-F\cdot D_{\epsilon}$$
implies that the increasing sequence of non-negative number $\{F\cdot D_{\epsilon}\}$ is bounded
above and hence converges. Therefore, the sequence of forms $\{\sigma_{D_{\epsilon}}\}$ converges smoothly to a
closed $J$-$(1,1)$ form $\sigma_{\mathcal{D}}$ satisfying
$$\sigma_{\mathcal{D}}\cdot\sigma_{\mathcal{D}}=0=\psi\cdot\sigma_{\mathcal{D}}, F\wedge\sigma_{\mathcal{D}}=cF^2$$
for some constant $c\geq0$. Since $[\sigma_{\mathcal{D}}]=0$ in $H_J^+(M).$
It follows $\sigma_{\mathcal{D}}=c\sigma_0.$

If $c=0$, it follows from the fact $\{F\cdot D_{\epsilon}\}$ is non-negative and increasing that $F\cdot D_{\epsilon}=0$ for
all $\epsilon$, in this case $D_{\epsilon}=0$ for all $\epsilon$ and therefore
$$\psi+{\mathcal{D}}^+_J(f)\geq -\epsilon F$$
 as required.

If $c\neq0$, the identity $\psi\cdot \sigma_0=0$ and Proposition \ref{0112} (for complex surfaces, see \cite[Proposition 5]{Bu})
imply that $\psi+{\mathcal{D}}^+_J(f)$ is a non-negative multiple of $\sigma_0$ for some smooth function $f$. By the hypothesis, there
is a non-zero integral effective $J$-curve $E$ on $M$ such that $E\cdot E=0$.
Since $[\sigma_0]=0$ in $H_J^+(M)$, it follows that $\sigma_0\cdot E=0$ and by Proposition \ref{0112} again, that $\sigma_0$ is a positive multiple
of $\sigma_E$, in this case, the desired result follows from Lemma \ref{55}.

This completes the proof of Lemma \ref{56}.
\ep

Now return to the proof of Theorem \ref{th53}. Let  $\psi\in(\wedge_J^{1,1}\otimes L^2(M))_w^0$ satisfying the hypothesis of Theorem \ref{th53}. By the proof
of the first part of Theorem \ref{th51} (for complex surfaces, see \cite[Theorem 14]{Bu}), there is a form $u\in\Omega_J^{0,1}(M)$ such that
$$\widetilde{\varphi}:=\varphi+d_J^+(u+\bar{u})$$ is positive (the hypothesis that $h_J^-=b^+-1$ in that theorem is only in the final sentence of the proof).

By Proposition \ref{0112}, $\widetilde{\varphi}\cdot\varphi$ is strictly positive. Let $t_0$ be the smaller solution of the equation
$$(\varphi-t_0\widetilde{\varphi})\cdot(\varphi-t_0\widetilde{\varphi})=0,$$
and set
$$\psi:=\varphi-t_0\widetilde{\varphi}.$$
Since $$(\varphi-t\widetilde{\varphi})\cdot(\varphi-t\widetilde{\varphi})>0,$$
for $t$ satisfying $0\leq t<t_0$, the sign of $F\cdot (\varphi-t\widetilde{\varphi})$ can not change for such $t$, so $F\cdot \psi\geq0$.
Since
$$(\varphi-\widetilde{\varphi})\cdot(\varphi-(\varphi-\widetilde{\varphi}))=-2|\overline{\partial}_Ju|^2\leq0,$$
it follows that $t_0\leq1$ and therefore for any effective $J$-curve $E\subset (M,J)$,
$$\psi\cdot E=(1-t_0)\varphi\cdot E\geq0.$$

The form $\psi$ therefore satisfies the hypothesis of Lemma \ref{56}. Applying Lemma \ref{56}, given $\epsilon>0$, there
is a smooth function $f_\epsilon$ such that
$$\psi+{\mathcal{D}}^+_J(f_{\epsilon})\geq -\epsilon F,$$
so if $\epsilon$ is chosen so small that $$t_0\widetilde{\varphi}-\epsilon F>0,$$ it follows that $$\varphi+{\mathcal{D}}^+_J(f_{\epsilon})>0$$ as required.

This completes the proof of Theorem \ref{th53}.
\begin{rem}
On closed almost Hermitian $4$-manifolds, ${\mathcal{D}}^+_J$ (resp. $\widetilde{{\mathcal{D}}}^+_J$) operator is defined by use of Lejmi operator
$$P=d_J^-d^*:\Omega_J^-(M)\rightarrow\Omega_J^-(M).$$
On $2n$-dimensional closed almost K\"{a}hler manifolds, Tan, Wang and Zhou define a generalized operator of $P$ in \cite{TaWaZh0}. It is
natural to ask the following question:

Find a generalized operator of ${\mathcal{D}}^+_J$ on a $2n$-dimensional closed  almost K\"{a}hler manifolds for $n\geq3$.

It is quite possible, by using generalized  ${\mathcal{D}}^+_J$ operator, one can study $2n$-dimensional almost complex manifolds.
\end{rem}

\section{Appendix: $(\mathcal{W},d_J^-)$-problem}
 \setcounter{equation}{0}
 This section is devoted to the $(\mathcal{W},d_J^-)$-problem and  $\widetilde{(\mathcal{W}},d_J^-)$-problem \cite[A.3]{TaWaZhZh} as $\bar{\partial}$-problem  in classical complex analysis \cite{ Ho0,Ho}.

Suppose that $(M,g,J,F)$ is a closed almost Hermitian $4$-manifold. In Section 2, we define the operator
   $$\mathcal{W}: L^2_2(M)_0\longrightarrow\Lambda^1_\mathbb{R}\otimes L^2_1(M)$$
   $$\mathcal{W}(f)=Jdf+d^*(\eta^1_f+\overline{\eta}^1_f),$$
where $\eta^1_f\in\Lambda^{0,2}_J\otimes L^2_2(M)$ satisfies $$d^-_JJdf+d^-_Jd^*(\eta^1_f+\overline{\eta}^1_f)=0.$$
Define
$$\mathcal{D}^+_J: L^2_2(M)_0\longrightarrow\Lambda^{1,1}_J\otimes L^2(M),$$
$$\mathcal{D}^+_J(f)=d\mathcal{W}(f).$$
or if $F$ is weakly $\widetilde{\mathcal{D}}_J^+$-closed (that is, $F$ is $J$-$(1,1)$ component of a $J$-taming symplectic form $\omega=F+d^-_J(v+\bar{v})$,
for $v\in\Omega_J^{0,1}(M)$), define the operator
  $$\mathcal{\widetilde{W}}: L^2_2(M)_0\longrightarrow\Lambda^1_\mathbb{R}\otimes L^2_1(M):$$
  $$\mathcal{\widetilde{W}}(f)=\mathcal{W}(f)-*_g(df\wedge d^-_J(v+\bar{v}))+d^*(\eta^2_f+\overline{\eta}^2_f),$$
where $\eta^2_f\in\Lambda^{0,2}_J\otimes L^2_2(M)$ satisfies
$$-d_J^-*_g(df\wedge d_J^-(v+\bar{v}))+d_J^-d^*(\eta^2_f+\bar{\eta}^2_f)=0.$$
Obviously,
   $$d^*\mathcal{\widetilde{W}}(f)=0, \,\,\, d^-_J\mathcal{\widetilde{W}}(f)=0.$$
Define $$\widetilde{\mathcal{D}}^+_J: L^2_2(M)_0\longrightarrow\Lambda^{1,1}_J\otimes L^2(M),$$
$$\widetilde{\mathcal{D}}^+_J(f)=d\mathcal{\widetilde{W}}(f).$$

As $\bar{\partial}$-problem in classical complex analysis,  $(\mathcal{W},d_J^-)$-problem (or $\widetilde{(\mathcal{W}},d_J^-)$-problem)
is whether $\mathcal{W}(f)=A$ (or $\widetilde{\mathcal{W}}(f)=A$) has a solution. Note that
$$d_J^-\mathcal{W}(f)=0\ ({\rm or} \  d_J^-\widetilde{\mathcal{W}}(f)=0).$$
If we can use the theory of Hilbert space, considering
\begin{align}\label{A61}
L_2^2(M)_0\stackrel{\mathcal{W}}\rightarrow \wedge_{\mathbb{R}}^1\otimes L_1^2(M)\stackrel{d_J^-}\rightarrow \wedge_J^-\otimes L^2(M),
\end{align}
or
\begin{align}\label{A62}
L_2^2(M)_0\stackrel{\widetilde{\mathcal{W}}}\rightarrow \wedge_{\mathbb{R}}^1\otimes L_1^2(M)\stackrel{d_J^-}\rightarrow \wedge_J^-\otimes L^2(M),
\end{align}
then the above problem is equivalent to whether the kernel of $d_J^-$ is equal to the image of $\mathcal{W}$ (or $\widetilde{\mathcal{W}}$).
Similar to $\bar{\partial}$-problem, we call this problem the $(\mathcal{W},d_J^-)$-problem (or $\widetilde{(\mathcal{W}},d_J^-)$-problem).
Our approach is along the line used by H\"{o}rmander to present the method of $L^2$ estimates for the $\bar{\partial}$-problem in \cite{Ho0,Ho}.
We summarize the above discussion in terms of the model of Hilbert spaces below:
$$H_1\stackrel{T}\rightarrow H_2\stackrel{S}\rightarrow H_3,$$
where $H_1,H_2,H_3$ are all Hilbert spaces, and $T,S$ are linear closed and densely defined operators. Assume $S\circ T=0$, the problem
is whether $A\in ker S$, a solution to
$$Tf=A$$
exists. First, note that a simple fact $Tf=A$ is equivalent to
\begin{align}\label{A63}
(Tf,B)_{H_2}=(A,B)_{H_2},
\end{align}
for each $B$ belongs to some dense subset. In fact, $(Tf-A,B)_{H_2}=0$ for each $B$ belongs to some dense subset if and only if
$(Tf-A,H_2)_{H_2}=0$ if and only if $Tf=A$.

Let $T^*$ be an adjoint operator of $T$ in the sense of distributions. By the theory of functional analysis \cite{Yo},
$T^*$ is a closed operator and $(T^*)^*=T$ if and only if $T$ is closed.

From \eqref{A63}, $$(Tf,B)_{H_2}=(A,B)_{H_2},$$
for each $B$ belongs to some dense subset. If this dense subset is contained in $D_{T^*}$, then noticing
$$(Tf,B)_{H_2}=(f,T^*B)_{H_1},$$
\begin{align}\label{A64}
Tf=A\Leftrightarrow (Tf,B)_{H_2}=(A,B)_{H_2}\Leftrightarrow (f,T^*B)_{H_1}=(A,B)_{H_2},
\end{align}
for each $B$ belongs to some dense subset in $D_{T^*}$.
Let
$$T^*A\rightarrow (A,B)_{H_2}$$
 be a linear functional defined on a subset of
$$\{T^*A| A\in{\rm\ some\ dense\ subset
\ in\  D_{T^*}}\}.$$
If we can extend the above functional to a bounded linear functional on the entire $H_1$,
then application of Riesz Representation Theorem \cite{Yo} to \eqref{A64} will show that the problem
$$Tf=A$$
is solved. Recall that Riesz Representation Theorem \cite{Yo} states that if
$$\lambda:H\rightarrow \mathbb{C}$$
is bounded linear function on a Hilbert space $H$, then there exists $g\in H$ such that
$$\lambda(x)=(x,g)_H,$$
for each $x\in H$. Hence the main step is whether we can extend
$$T^*B\rightarrow (A,B)_{H_2}$$
to a bounded linear functional on the extire $H_1$ (for details, see \cite{Ho0,Ho}).

As in classical complex analysis, we have the following lemmas:
  \begin{lem}\label{A01}
      (cf. {\rm\cite[Theorem 1.1.1]{Ho0}},\cite[Lemma A.24]{TaWaZhZh} )
     If there exists a constant $C_A$ depending only on $A$ such that
    \begin{equation}\label{A001}
    |(A,B)_{H_2}|\leq C_A\|T^*B\|_{H_1},
   \end{equation}
    then $$T^*B\longrightarrow(A,B)_{H_2}$$ can be extended to a bounded linear functional on $H_1$.
      \end{lem}
In the above discussion, we used only the front half of
$$ H_1\stackrel{T}\rightarrow H_2\stackrel{S}\rightarrow H_3. $$
However, since we only need to solve the equation
$$Tf=A\  {\rm or\ } (T^*B,f)=(B,A)$$
for $A\in \ker S$, it is necessary to prove \eqref{A001} for $A\in H_2$, rather we just need to prove
 \eqref{A001} for $A\in \ker S$. In this case, we hope that $B$ in \eqref{A001} belongs to some dense
 subsets in $D_{T^*}$.

The method of proving
$$|(A,B)_{H_2}|\leq C_A\|T^*B\|_{H_1}$$
is through proving a more general inequality
$$\|B\|_{H_2}^2\leq C(\|T^*B\|_{H_1}^2+\|SB\|_{H_3}^2),$$
for $B\in D_{T^*}\cap D_S$.

First we note, in our problem $D_{T^*}$ and $D_S$ contain the space of smooth functions on $M$ which is a closed
$4$-manifold, hence $D_{T^*}\cap D_S$ is dense on both $D_{T^*}$ and $H_2$. Notice that
$$T=\mathcal{W}\ ({\rm or} \ \widetilde{\mathcal{W}}), S=d_J^-,$$
$M$ being a closed almost Hermitian $4$-manifold (or a closed tamed $4$-manifold), it is easy to see that
$$D_{{\mathcal{W}}^*}\cap D_{d_J^-} ({\rm\  or \ } D_{{\widetilde{\mathcal{W}}}^*}\cap D_{d_J^-} )$$
is dense in the space $H_2=\wedge_{\mathbb{R}}^1\otimes L^2(M)$. In fact, $\Omega^1(M)$ is dense in
$$D_{{\mathcal{W}}^*}\cap D_{d_J^-} ({\rm\  or \ } D_{{\widetilde{\mathcal{W}}}^*}\cap D_{d_J^-} )$$
and $H_2$. For $(\mathcal{W},d_J^-)$-problem (or $\widetilde{(\mathcal{W}},d_J^-)$-problem),
$$H_1:=L^2(M)_0\  {\rm and} \  H_3:=\wedge_J^-\otimes L^2(M).$$

   \begin{lem}\label{A02}
  (cf. {\rm\cite[Theorem 1.1.2]{Ho0},\cite[Lemma A.25]{TaWaZhZh}})
  If \begin{equation}\label{A002}
      \|B\|^2_{H_2}\leq C(\|T^*B\|^2_{H_1}+\|SB\|^2_{H_3}),\  B\in D_{T^*}\cap D_S,
   \end{equation}
 then
  \begin{equation}\label{A003}
    |(A,B)_{H_2}|\leq C^{\frac{1}{2}}\|A\|_{H_2}\|T^*B\|_{H_1}\,\,\,\forall A\in\ker S,\,\,\, B\in D_{T^*}\cap D_S.
  \end{equation}
   \end{lem}
 Applying Lemma \ref{A02}, we have that if
 $$\|B\|_{H_2}^2\leq C(\|T^*B\|_{H_1}^2+\|SB\|_{H_3}^2),$$
 for $B\in D_{T^*}\cap D_S$,
 then
   \begin{equation*}
    |(A,B)_{H_2}|\leq C^{\frac{1}{2}}\|A\|_{H_2}\|T^*B\|_{H_1}\,\,\,\forall A\in\ker S,\,\,\, B\in D_{T^*}\cap D_S.
  \end{equation*}
 Hence, by Lemma \ref{A01}, $$T^*B\rightarrow (A,B)_{H_2}$$ can be extended to a bounded linear functional
 on $H_1$, whose bound is $C^{\frac{1}{2}}\|A\|_{H_2}$. By Riesz Representation Theorem \cite{Yo}, there exists $f\in H_1$ such that
 $$(T^*B,f)_{H_1}=(B,A)_{H_2},$$
for each $B\in D_{T^*}\cap D_S$. Since $D_{T^*}\cap D_S$ is dense in $H_2$, we have
 $$(B,Tf)_{H_2}=(B,A)_{H_2},$$
 for each $B\in H_2$. By \eqref{A64}, the equation $Tf=A$ has a solution. In addition, from Riesz Representation Theorem \cite{Yo}, we have
 $$\|f\|_{H_1}\leq C^{\frac{1}{2}}\|A\|_{H_2},$$
for $f\in(\ker T)^{\perp}$.
$$\|f\|_{H_1}\leq C^{\frac{1}{2}}\|A\|_{H_2}$$
is the direct consequence of Riesz Representation Theorem \cite{Yo}. To show $f\in(\ker T)^{\perp}$, note that, according to the way that
$$T^*B\rightarrow (B,A)_{H_2}$$
 is extended to a bounded linear functional on the entire $H_1$, this functional vanishing on the orthogonal component of
 $\{T^*B|B\in D_{T^*}\}$, thus, $f\in \{T^*B|B\in D_{T^*}\}$. If $$f\in \lim_{k\rightarrow\infty} T^*B_k,$$
 then for every $x\in \ker T$, we have
 $$(x,f)_{H_1}=\lim_{k\rightarrow\infty}(x,T^*B_k)_{H_1}=\lim_{k\rightarrow\infty}(Tx,B_k)_{H_2}=0,$$
 hence, $f\in(\ker T)^{\perp}$.

 In general, the solution of $Tf=g$ is not unique. In fact, if $f_1\in\ker T$, then
   \begin{align*}
 (T^*B,f+f_1)_{H_1}&=(T^*B,f)_{H_1}+(T^*B,f_1)_{H_1}\\
                   &=(T^*B,f)_{H_1}+(B,Tf_1)_{H_2}\\
                   &=(T^*B,f)_{H_1},
   \end{align*}
 and $f$, $f+f_1$ are both the solution of $Tf=A$. However, $f\in(\ker T)^{\perp}$ is the condition to assure that
 the above solution to $Tf=A$ is unique.

  From the above discussion, we have:
   \begin{lem}\label{A03}
 (cf. \cite[Theorem 1.1.4]{Ho0} \cite[Lemma A.26]{TaWaZhZh})
  If $$\|B\|^2_{H_2}\leq C(\|T^*B\|^2_{H_1}+\|SB\|^2_{H_3}),$$ then $Tf=A$ has a solution for $A\in \ker S$.
  This solution $f$ satisfies the estimate
  \begin{equation}\label{inequvi5}
    \|f\|_{H_1}\leq C^{\frac{1}{2}}\|A\|_{H_2},\,\,\,f\in(\ker T)^{\bot}.
  \end{equation}
   \end{lem}
In our application of the above lemmas, the spaces $H_1$, $H_2$ and $H_3$ will be $L^2(M)_0$, $\wedge_{\mathbb{R}}^1\otimes L^2(M)$
and $\wedge_J^-\otimes L^2(M)$, respectively, $T$ the operator $\mathcal{W}$ (or $\widetilde{\mathcal{W}}$), and let $G$ be the set of
all $A\in \wedge_{\mathbb{R}}^1\otimes L_1^2(M)$ with $d_J^-A=0$. Let $S$ be the operator from $\wedge_{\mathbb{R}}^1\otimes L_1^2(M)$
to $\wedge_J^-\otimes L^2(M)$ defined by $d_J^-$. Then $G$ is the null space of $S$, and to prove \eqref{A001} it will be sufficient to show
that
 \begin{equation}\label{A09}
      \|A\|^2_{H_2}\leq c(\|T^*A\|^2_{H_1}+\|SA\|^2_{H_3}), \,\,\, A\in D_{T^*}\cap D_S.
   \end{equation}   

In order to study $(\mathcal{W},d_J^-)$ (or $\widetilde{(\mathcal{W}},d_J^-)$)-problem, we need the following lemma:
  \begin{lem}\label{A04}
   Let $(M,g,J,F)$ be a closed almost Hermitian $4$-manifold.\\
   (1) For any $A\in\Omega_{\mathbb{R}}^1(M)$ satisfying $d_J^-A=0$, then
   $$\mathcal{W}^*(A)=\frac{2(d_J^+A\wedge F+A\wedge dF)}{F^2}.$$
   (2) If $F$ is weakly $\widetilde{\mathcal{D}}_J^+$-closed and any $A\in \Omega_{\mathbb{R}}^1(M)$ satisfies $d_J^-A=0$,
   then
    $$\widetilde{\mathcal{W}}^*(A)=\frac{2d_J^+A\wedge F}{F^2}.$$
  \end{lem}
 \bp (1)
 By Definition \ref{0103}, we have
  $$\mathcal{{W}}(f)=Jdf+d^*(\eta^1_f+\bar{\eta}^1_f)$$
  satisfying
  $$d^-_J\mathcal{{W}}(f)=0.$$
Thus,
  $$d\mathcal{{W}}(f)=d^+_J\mathcal{{W}}(f)\in\Lambda^{1,1}_J\otimes L^2(M).$$
  Note that $Jdf=-*_g(d f\wedge F)$.
  We  assume that if
  $A\in \Omega^1_{\mathbb{R}}(M)$ and $d^-_JA=0,$
  then
   \begin{eqnarray*}
    (\mathcal{{W}}(f),A) &=&(-*_g(d f\wedge F)+d^*(\eta^1_f+\bar{\eta}^1_f),A)\\
    &=&(-*_g(d (fF)-fdF)+d^*(\eta^1_f+\bar{\eta}^1_f),A)\\
     &=&(-*_g(d (fF)-fdF)),A)\\
      &=&(d^*(fF),A)+(*_g(fdF),A)\\
      &=&((fF),dA)+(*_g(fdF),A)\\
     &=& \int_M fF\wedge dA-fdF\wedge A\\
     &=& (f, \mathcal{W}^*A).
  \end{eqnarray*}
  Thus, the formal $L^2$-adjoint operator of $\mathcal{{W}}$ is
  \begin{equation}
   \mathcal{W}^*(A)=\frac{2(d_J^+A\wedge F+A\wedge dF)}{F^2},
  \end{equation}
for $d_J^-A=0$.

(2) If $F$ is weakly $\widetilde{\mathcal{D}}_J^+$ closed, then $F$ is $J$-$(1,1)$ component of a symplectic form
$$\omega_1=F+d_J^-(v+\bar{v}),$$
where $v\in \Omega_J^{0,1}(M)$. If $f\in C^\infty(M)_0$, where $C^\infty(M)_0$ is the set of smooth function with
$$\int_M f dvol_g=0;$$
$A\in\Omega_{\mathbb{R}}^1(M)$ satisfying $d_J^-A=0$. Then, by Definition \ref{0104}, we have
  \begin{align*}
 (\widetilde{\mathcal{W}}(f),A)_{L^2(M)}&=\int_MA\wedge d(f\omega_1+\eta_f^1+\eta_f^2+{\bar{\eta}}_f^1+{\bar{\eta}}_f^2)\\
                   &=\int_MdA\wedge (f\omega_1+\eta_f^1+\eta_f^2+{\bar{\eta}}_f^1+{\bar{\eta}}_f^2)\\
                   &=\int_Md_J^+A\wedge F\\
                   &=(f,\widetilde{\mathcal{W}}^*A)_{L^2(M)}.
   \end{align*}
Thus, if $d_J^-A=0$, the formally adjoint operator $\widetilde{\mathcal{W}}$ is
 \begin{align}
 \widetilde{\mathcal{W}}^*(A)=\frac{2d_J^+A\wedge F}{F^2}.
    \end{align}
\ep

 By Lemmas \ref{A01}-\ref{A04}, we have the following lemma for solving $(\mathcal{W},d_J^-)$ (or $\widetilde{(\mathcal{W}},d_J^-)$)-problem:
   \begin{lem}\label{65}
(1)   Let $(M,F,J,g_J)$ be a closed almost Hermitian $4$-manifold, $F$ weakly $\mathcal{D}_J^+$-closed.
   Suppose that $\psi\in\wedge_J^{1,1}\otimes L^2(M)$ is $d$-exact.
Then  $\psi$ is ${\mathcal{D}}_J^+$-exact, that is, there exists $f_\psi\in L_2^2(M)_0$ such that $\psi={\mathcal{D}}_J^+(f_\psi)$.

(2)  Let $(M,F,J,g_J)$ be a closed almost Hermitian $4$-manifold which is tamed by $\omega_1=F+d_J^-(v+\bar{v})$, where $v\in\Omega_J^{0,1}(M)$.
   Suppose that $\psi\in\wedge_J^{1,1}\otimes L^2(M)$ is $d$-exact.
Then  $\psi$ is $\widetilde{\mathcal{D}}_J^+$-exact, that is, there exists $f_\psi\in L_2^2(M)_0$ such that $\psi=\widetilde{\mathcal{D}}_J^+(f_\psi)$.
  \end{lem}

  \begin{rem}\label{A6}
For compact complex surface $(X,J)$, Buchdahl \cite[Lemma 10]{Bu} proved that if $b^1(X)$ is even (that is, $h_J^-=b^+-1$) and if
$\psi\in\wedge_J^{1,1}\otimes L^2(M)$ is $d$-exact, then $\psi=\overline{\partial}_J\partial_J f$ for some $f\in L_2^2(X)$.
In \cite[Appendix A]{TaWaZhZh}, one can solve $\widetilde{(\mathcal{W}},d_J^-)$)-problem on compact $J$-pseudoconvex domain.
 \end{rem}
{\bf Proof of Lemma \ref{65}.} It is easy to see that $C^{\infty}(M)_0$ is a dense subset of $L_2^2(M)_0$, and $\Omega_{\mathbb{R}}^1(M)$ is a
dense subset $\wedge_{\mathbb{R}}^{1}\otimes L_1^2(M)$. For any $A\in\wedge_{\mathbb{R}}^{1}\otimes L_1^2(M)$ satisfying $d_J^-A=0$. Thus,
 there exists a sequence $\{A_i'\}\subset \Omega_{\mathbb{R}}^{1}(M)$ satisfying
$A_i'\rightarrow A$ in $L_1^2$ space. Let $d_J^-A_i=\sigma_i\in\Omega_J^-(M)$, thus, $\sigma_i\rightarrow 0$ in $L^2$ as $i\rightarrow\infty$.
By Lemma \ref{2130}, there exists $\eta_i\in\Omega_J^{0,1}(M)$ such that
$$\sigma_i+d_J^-d^*(\eta_i+\bar{\eta}_i)=0.$$
Hence, let $$A_i:=A_i'+d^*(\eta_i+\bar{\eta}_i),$$ then $d_J^-A_i=0$ and
$A_i\rightarrow A$ in $L_1^2$ space as $i\rightarrow\infty$.
Therefore, $\ker d_J^-\cap\Omega_{\mathbb{R}}^{1}(M)$ is dense subset of $\ker d_J^-$.

Suppose that $F$ is weakly $\widetilde{\mathcal{D}}_J^+$-closed. There exists a $J$-taming symplectic form $\omega_1=F+d_J^-(v+\bar{v})$, where $v\in\Omega_J^{0,1}(M)$. Without loss of generality, assume that
$$\int_MF^2=2,$$
by Lemma \ref{A04}, we may assume that $A\in \Omega_{\mathbb{R}}^1(M)$ satisfies $d_J^-A=0$.
Then
$$\widetilde{\mathcal{W}}^*(A)=\frac{2d_J^+A\wedge F}{F^2}=\wedge d_J^+A.$$
It follows that
\begin{align*}
&\int_M\widetilde{\mathcal{W}}^*(A)d vol_g=\int_Md_J^+A\wedge F\\
&=\int_MdA\wedge F=\int_MdA\wedge \omega_1=0.
\end{align*}
Hence
$\widetilde{\mathcal{W}}^*(A)\in L^2(M)_0$. Note that
\begin{align*}
0<\|\widetilde{\mathcal{W}}^*(A)\|_{L^2}^2=\int_M(\wedge d_J^+A)^2dvol_g
\leq C \ \| \wedge d_J^+A\|_{L^2}^2< \infty.
\end{align*}
Hence, if $d_J^-A=0$, for any $f\in L^2(M)_0$, then
\begin{align*}
(f,\widetilde{\mathcal{W}}^*(A))_{L^2}\leq \|f\|_{L^2}\cdot\|\wedge d_J^+A\|_{L^2}.
\end{align*}
Hence,
\begin{equation*}
  \|\widetilde{\mathcal{W}}^*(A)\|_{L^2} = \mathop{\sup_{h \in L^2(M)_0}}_{\|h\|_{L^2}} (h,\widetilde{\mathcal{W}}^*(A))_{L^2} \leq \|\wedge d_J^+A\|_{L^2}
\end{equation*}

By Lemma \ref{A01}-\ref{A03}, it follows that there exists a unique  $f_A\in C^{\infty}(M)_0$
such that $\widetilde{\mathcal{W}}(f_A)=A$. Hence, we solve $(\widetilde{\mathcal{W}},d_J^-)$-problem.

If $F$ is weakly ${\mathcal{D}}_J^+$-closed, then for any  $A\in\Omega_{\mathbb{R}}^{1}$ satisfying $d_J^-A=0$
we have
$${\mathcal{W}}^*(A)=\frac{2d_J^+A\wedge F+A\wedge dF}{F^2}.$$
Hence
\begin{align*}
&\int_M{\mathcal{W}}^*(A)d vol_g=\int_Md_J^+A\wedge F+A\wedge dF\\
&=\int_Md(A\wedge F)=0,
\end{align*}
${\mathcal{W}}^*(A)\in L^2(M)_0$.

For any $A\in\Omega_{\mathbb{R}}^{1}(M)$, by gauge theory \cite{DoKr}, there exists a unique $\widetilde{A}\in\Omega_{\mathbb{R}}^{1}(M)$
such that $d^*\widetilde{A}=0$ and $d\widetilde{A}=dA$. It is easy to see that
\begin{align*}
\|\widetilde{A}\|_{L^2}^2\leq C\|d\widetilde{A}\|_{L^2}^2\leq C(\|d_J^+A\|_{L^2}^2+\|d_J^-\widetilde{A}\|_{L^2}^2),
\end{align*}
where $C$ is a constant. Hence, for any $A\in\Omega_{\mathbb{R}}^{1}(M)$ satisfying $d_J^-A=0$, then
 there exists a unique $\widetilde{A}\in\Omega_{\mathbb{R}}^{1}(M)$ such that $d^*\widetilde{A}=0$, $d\widetilde{A}=dA$, and
 $${\mathcal{W}}^*(\widetilde{A})=\frac{2d_J^+\widetilde{A}\wedge F+\widetilde{A}\wedge dF}{F^2},$$
${\mathcal{W}}^*(\widetilde{A})\in C^\infty(M)_0$, for any $f\in C^{\infty}(M)_0$.
\begin{align*}
(f,\mathcal{W}^*\widetilde{A})_{L^2}\leq \|f\|_{L^2}\cdot \|\mathcal{W}^*\widetilde{A}\|_{L^2}\\
\leq C\|f\|_{L^2}\cdot \|d\widetilde{A}\|_{L^2}=C\|f\|_{L^2}\cdot \|d_J^+A\|_{L^2}.
\end{align*}
Similarly, by Lemma \ref{A01}-\ref{A03}, there exists a unique $f_A\in C^\infty(M)_0$ such that
$$\mathcal{W}(f_A)=\widetilde{A},$$
$$d_J^-\mathcal{W}(f_A)=d_J^-A=0$$
and
$$d\mathcal{W}(f_A)=d_J^+A\in \Omega_{J}^{+}(M).$$
Thus, we can solve $(\mathcal{W},d_J^-)$-problem.

 This completes the proof Lemma \ref{65}. \medskip

Finally, we give a characterization of $(\Lambda^{1,1}_J\otimes L^2(M))^0_{{w}}$ or $(\Lambda^{1,1}_J\otimes L^2(M))^0_{{\widetilde{w}}}$ by the following lemma:
 \begin{lem}(cf. \cite[Lemma 3.11]{TaWaZhZh})\label{67}
Suppose that  $(M,g,J,F)$ is a closed almost Hermitian $4$-manifold.  If $F$ is weakly ${\mathcal{D}}_J^+$ (or $\widetilde{\mathcal{D}}_J^+$)-closed
with $h^-_J=b^+$ (or $h^-_J=b^+-1$), then
  $$(\Lambda^{1,1}_J\otimes L^2(M))^0_{{w}}=(\Lambda^{1,1}_J\otimes L^2(M))_{{w}}$$
 $$( or\ \ (\Lambda^{1,1}_J\otimes L^2(M))^0_{\widetilde{w}}=(\Lambda^{1,1}_J\otimes L^2(M))_{\widetilde{w}}).$$
 \end{lem}
\bp
If $F$ is weakly ${\mathcal{D}}_J^+$-closed and $h^-_J=b^+$, then it is easy to see that
$$\dim (H_g^+\cap H_J^+)=0$$
and
 $$(\Lambda^{1,1}_J\otimes L^2(M))^0_{{w}}=(\Lambda^{1,1}_J\otimes L^2(M))_{{w}}.$$
If $F$ is weakly $\widetilde{\mathcal{D}}_J^+$-closed and $\omega_1=F+d_J^-(v+\bar{v})$, where $v\in\Omega_J^{0,1}(M)$
then
$$\dim (H_g^+\cap H_J^+)=1.$$
It is easy to see that
 $$ (\Lambda^{1,1}_J\otimes L^2(M))^0_{\widetilde{w}}=(\Lambda^{1,1}_J\otimes L^2(M))_{\widetilde{w}}.$$
\ep

 \vspace{3mm}\par\noindent {\bf{Acknowledgements.}}
The first author would like to thank Professor Ping Li, Professor Fangyang Zheng and Professor Nichola Buchdahl for useful comments and suggestions. The second
author would like to thank Professor Fan Ding and Professor Zhi L\"{u} for their support and guidance.

\textsc{Hongyu Wang}, \\
School of Mathematical Sciences,\\
Yangzhou University,\\
 Yangzhou,  225002,\\
 Jiangsu, People's Republic of China\\
 e-mail: \verb"hywang@yzu.edu.cn"\vspace{3mm}

 \textsc{Ken Wang}, \\
School of Mathematical Sciences,\\
Fudan University,\\
Shanghai,  100433,\\
 People's Republic of China\\
 e-mail: \verb"kanwang22@m.fudan.edu.cn"\vspace{3mm}

\textsc{Peng Zhu}, \\
School of Mathematics and Physics,\\
 Jiangsu University of Technology,\\
 Changzhou,  213001,\\
 Jiangsu, People's Republic of China\\
 e-mail: \verb"zhupeng@jsut.edu.cn" 
\end{document}